\documentclass[hidelinks,onefignum,onetabnum,final]{siamart220329}
\usepackage{amsmath,amsfonts,amssymb}
\usepackage{latexsym}
\usepackage{color}
\usepackage{wrapfig}
\usepackage{mathtools} 

\let\eps\varepsilon
\newcommand{\R}{\mathbb R}

\newcommand{\cA}{\mathcal A}
\newcommand{\hcA}{{\widehat{\mathcal A}}}

\newcommand{\ba}{\mathbf a}

\newcommand{\blf}{\mathbf f}
\newcommand{\bn}{\mathbf n}

\newcommand{\bs}{\mathbf s}
\newcommand{\bt}{\mathbf t}

\newcommand{\bu}{\mathbf u}

\newcommand{\te}{\widetilde e}
\newcommand{\bU}{\mathbf U}

\newcommand{\bj}{\mathbf j}

\newcommand{\bx}{\mathbf x}

\newcommand{\bchi}{{\boldsymbol{\chi}}}

\newcommand{\cT}{\mathcal T}

\newcommand{\cL}{\mathcal L}
\newcommand{\balpha}{\boldsymbol{\alpha}}
\newcommand{\hbalpha}{\widehat{\boldsymbol{\alpha}}}

\newcommand{\bbeta}{\boldsymbol{\beta}}

\newcommand{\tsigma}{\widetilde{\sigma}}
\newcommand{\halpha}{\widehat{\alpha}}
\newcommand{\bhalpha}{\widehat{\boldsymbol{\alpha}}}

\newcommand{\hbs}{\widehat{\bs}}
\newcommand{\bseta}{\boldsymbol{\eta}}

\newcommand{\bPhi}{\boldsymbol{\Phi}}
\newcommand{\bPsi}{\boldsymbol{\Psi}}

\newcommand{\tbPhi}{\widetilde{\bPhi}}

\newcommand{\rA}{\mathrm A}

\newcommand{\rI}{\mathrm I}
\newcommand{\rK}{\mathrm K}
\newcommand{\rM}{\mathrm M}
\newcommand{\rL}{\mathrm L}
\newcommand{\rP}{\mathrm P}

\newcommand{\rS}{\mathrm S}
\newcommand{\rU}{\mathrm U}
\newcommand{\rV}{\mathrm V}
\newcommand{\rW}{\mathrm W}

\newtheorem{remark}{Remark}
\def\norm#1{\|\hskip-0.25ex|{#1}|\hskip-0.25ex\|}
\def\rev#1{{\color{blue}#1}}

\def\MO#1{{\color{red}#1}}

\begin{document}
\title{A priori  analysis of a tensor ROM for parameter dependent parabolic problems}
\author{Alexander V. Mamonov\thanks{Department of Mathematics, University of Houston, 
		Houston, Texas 77204 (avmamonov@uh.edu).} \and
	Maxim A. Olshanskii\thanks{Department of Mathematics, University of Houston, 
		Houston, Texas 77204 (maolshanskiy@uh.edu).}
	}
\maketitle

\begin{abstract}
A space--time--parameters structure of parametric parabolic PDEs motivates the application of tensor methods to define  reduced order models (ROMs).
Within a tensor-based ROM framework, the matrix SVD -- a traditional dimension reduction technique -- yields to   a low-rank tensor decomposition (LRTD). 
 Such tensor extension of the  Galerkin proper orthogonal decomposition ROMs (POD--ROMs) benefits both the practical efficiency of the ROM and its amenability for rigorous error analysis when applied to parametric PDEs. The paper addresses the error analysis of the Galerkin  LRTD--ROM for an abstract linear parabolic problem that depends on multiple physical parameters. An error estimate for the LRTD--ROM solution is proved, which is uniform with respect to problem parameters and extends to parameter values not in a sampling/training set. The estimate is given in terms of discretization and sampling mesh properties,  and LRTD accuracy. The estimate depends on the {local} smoothness rather than on the Kolmogorov {$n$}-widths of the parameterized manifold of solutions.
 Theoretical results are illustrated with several numerical experiments.   
\end{abstract}


\section{Introduction} 
Numerical analysis of traditional discretization methods for PDEs is a mature research field. However, extending many of its fundamental results to reduced-order computational models appears to be a challenging task.

For projection-based ROMs such as Galerkin POD-ROMs, the complexity of this task is particularly determined by the {fact that} approximation properties of the reduced {spaces are problem} and metric-dependent. While Galerkin POD-ROMs still lack a complete error analysis that fully explains their practical performance,  progress has been made over the last two decades, starting with a milestone paper by Kunisch and Volkwein~\cite{kunisch2001galerkin}, which provides an initial error analysis of POD for certain parabolic problems. The analysis from \cite{kunisch2001galerkin} has been extended and modified in various ways, and error estimates for POD have been derived for incompressible fluid model equations~\cite{kunisch2002galerkin, iliescu2014variational, xie2018numerical, garcia2023pod}, parametric elliptic PDEs~\cite{kahlbacher2007galerkin}, and nonlinear dynamical systems~\cite{hinze2005proper}; see also \cite{luo2009mixed, chapelle2012galerkin, iliescu2014snapshot, kostova2018model, koc2021optimal}.

These analyses have primarily addressed the ability of POD and POD-ROMs to approximate solution states used to generate the POD basis. An {a priori} error analysis for the case of out-of-sample data, such as varying problem parameters or initial conditions, remains a largely open question. An important development for parametric problems, attempting to address this question, was the error analysis of POD-Greedy {and other greedy reduced basis} algorithms in terms of convergence rates of the Kolmogorov {$n$}-widths of the parameterized manifold of solutions~\cite{buffa2012priori, binev2011convergence, haasdonk2013convergence}.
The analysis assumes that the Kolmogorov {$n$}-widths of the solution manifold decay sufficiently fast and that the worst-error parameter search over the parametric domain is done accurately. In practice, the error in the greedy search is evaluated over a fine but finite subset of the parameter domain and with the help of {a posteriori} error bounds.

{Tensor methods, such as tensor cross approximations, have recently shown promising success in approximating solutions for parametric PDE problems; see, e.g.~\cite{schwab2011sparse,khoromskij2011tensor,ballani2015hierarchical,dolgov2015polynomial,nouy2017low,dolgov2019hybrid,glau2020low}.}
The application of tensor techniques also opens up new possibilities for both building reduced bases and performing error analyses of resulting Galerkin ROMs. The present paper contributes to the latter topic. Loosely speaking, the idea is that in the parametric setting, replacing SVD (or POD) with a low-rank tensor decomposition (LRTD) of the snapshot tensor allows us to retain and recover information about the variation of optimal reduced spaces with respect to parameters. This enables constructing parameter-specific reduced spaces for out-of-sample parameter values and developing the analysis by exploiting the {(local)} smoothness of the solution manifold, {while avoiding global characteristics such as Kolmogorov {$n$}-widths of the parametric solution manifold}. This also extends existing a priori error analyses of Galerkin ROMs to the parametric setting, covering out-of-sample parameter values.

The remainder of the paper is organized into six sections. Section~\ref{s:setup} formulates the problem and recalls a conventional POD approach to assist the reader in understanding the tensor ROM (or LRTD--ROM) as a generalization of the POD--ROM. Section~\ref{s:prelim} collects the necessary preliminaries of tensor linear algebra. The  LRTD--ROM is introduced in Section~\ref{s:TROM}. In Section~\ref{s:analysis}, we derive interpolation, approximation, and stability results that we need to further derive the error bound. The error estimate is proved in Section~\ref{s:error}. Section~\ref{s:num} illustrates the analysis with the results of several numerical experiments.

\section{Problem formulation}
\label{s:setup} 

We are interested in the error analysis of a  LRTD--ROM for parameter dependent parabolic problems. 
To formulate the problem, consider a bounded Lipschitz physical domain $\Omega\subset \mathbb{R}^d$ with $d=2,3$ and a compact parameter domain $\cA\subset \mathbb{R}^D$ with $D\ge 1$. 
Furthermore, we consider a linear parabolic equation: 
For a given $\balpha\in \cA$, find $u\in L^2(0,T;H^1_0(\Omega))$ such that  $u_t\in L^2(0,T;H^{-1}(\Omega))$ and
\begin{equation}
\label{eqn:GenericPDE}
u_t + \cL(\balpha)u =\blf_{\balpha},  \quad t \in (0,T), \quad \text{and}~ u|_{t=0} = u_0(\balpha),
\end{equation}
with a parameter dependent elliptic operator $\cL(\balpha):\, H^1_0(\Omega)\to H^{-1}(\Omega)$, right-hand side functional {$\blf_{\balpha}\in L^2(0,T;H^{-1}(\Omega))$}, and initial condition {$u_0(\balpha)\in L^2(\Omega)$}.
We assume homogeneous Dirichlet boundary conditions for $u$, which is not a restrictive simplification.   

It is convenient to define a bilinear form $a_{\balpha}(u,v)=\langle\cL(\balpha)u,v \rangle$ on $H^1_0(\Omega)\times H^1_0(\Omega)$, where the angle brackets denote {the} $L^2$-duality product for $H^{-1}\times H^1_0$. 
For the bilinear form we assume uniform ellipticity  and continuity:
\begin{equation}\label{a_cond}
	c_a\| v\|_1^2  \le a_{\balpha}(v,v),\qquad a_{\balpha}(u,v)\le C_a \|u\|_1\|v\|_1 \quad \forall\, u,v \in H^1_0(\Omega),
\end{equation}
with $\|v\|_1:=\|v\|_{H^1}$ and ellipticity and continuity constants, $c_a>0$, $C_a$, independent of $\balpha$. 

As our full-order model, we adopt a conforming finite element (FE) discretization of \eqref{eqn:GenericPDE}. However, other widely used discretization techniques would be equally suitable for the  LRTD--ROM. To this end, let's consider a shape-regular quasi-uniform triangulation $\cT_h$ of $\Omega$ with $h=\min_{T\in\cT_h}\mbox{diam}(T)$ and define the $H^1$-conforming finite element space of a polynomial degree $m$:
\[
V_h=\{v\in H^1_0(\Omega)\,:\, v|_T\in\mathbb{P}_m,~~\forall\,T\in\cT_h\},\qquad M=\mbox{dim}(V_h).
\]  
Let $u^0_h=I_h(u_0(\balpha))$, where $I_h:H^1_0(\Omega)\to V_h$ is an interpolation operator satisfying standard interpolation properties for {functions from the  Sobolev space $H^{m+1}(\Omega)$} (see, e.g., \cite{scott1990finite}).
A first order in time backward-Euler FE method reads:  For a given $\balpha\in \cA$, find $u_h^n\in V_h$ for $n=1,\dots,N$ solving
\begin{equation}
	\label{eqn:FEM}
	\Big(\frac{u_h^{n}-u_h^{n-1}}{\Delta t},v_h\Big)_0 + a_{\balpha}(u_h^n,v_h) =\blf_{\balpha}(v_h),\quad~   \forall\,v_h\in V_h,\quad \Delta t=T/N.
\end{equation}
Here further we use the shortcut notation for the $L^2(\Omega)$ inner product and norm, $(\cdot,\cdot)_0=(\cdot,\cdot)_{L^2(\Omega)}$ and $\|\cdot\|_0=\|\cdot\|_{L^2(\Omega)}$.

Let $t_n=n\Delta t$, $n=0,\dots,N$.
For smooth solutions to \eqref{eqn:GenericPDE} it is a textbook exercise to show the convergence estimate~\cite{thomee2007galerkin},
\begin{equation}\label{h_conv}
		\max_{n=0,\dots,N}\|u(t^n)-u^n_h\|^2_0 +\Delta t {\sum_{n=0}^{N}}
  \|u(t^n)-u^n_h\|^2_1\le C\,(\Delta t^2 + h^{2m}),
\end{equation}
with $C>0$  depending on ellipticity and continuity constants from \eqref{a_cond} and higher order Sobolev norms of the solution and its time derivatives. 
In particular, $C$ is   bounded independent of  $\balpha$ if the solution is uniformly bounded in the following sense: 
\begin{equation}\label{UnifEst}
	\sup_{\balpha\in\cA}\Big(\|u_0\|_{H^m(\Omega)} + \int_0^T(\|u_t\|_{H^{{m+1}}(\Omega)}+ \|u_{tt}\|_0)\,dt  \Big) < \infty.
\end{equation}
Hence, we assume the uniform regularity condition \eqref{UnifEst}. In Section~\ref{s:5.1}, we demonstrate that \eqref{UnifEst} holds if the data of the problem is sufficiently smooth.

\subsection{Conventional POD}
The conventional POD computes a representative collection 
of states,  referred to as snapshots, at times $t_j$ and for  selected values of parameters:
{\begin{equation}
u_h^j(\widehat\balpha_k) \in V_h, 
\quad j = 1,\ldots, N, \quad k = 1,\ldots,K.
\end{equation}
Here the snapshots $u_h^j(\widehat\balpha_k)$ 
denote the solutions of the full order model~\eqref{eqn:FEM} for
$\balpha = \widehat\balpha_k$, the parameter 
samples from the parameter domain
${\hcA} := \{\widehat\balpha_1, \dots, \widehat\balpha_K \} \subset \cA$.}
Next, using eigenvalues and eigenvectors of a snapshot {Gramian} matrix~\cite{holmes2012turbulence}, one finds a low-dimensional POD basis $\psi_h^i\in V_h$, $i=1,\dots,\ell\ll M$, that solves the minimization problem
\begin{equation}\label{pod_minim}
\begin{split}
&\min_{\{\psi_h^i\}_{i=1}^\ell} 
\sum_{j=1}^N\sum_{k=1}^{K}\Big\|
{u_h^j(\widehat\balpha_k)} - \sum_{i=1}^{\ell}\big(\psi_h^i,{u_h^j(\widehat\balpha_k)}\big)_\ast\psi_h^i\Big\|_\ast^2, \\
&\text{subject to }~ \big(\psi_h^i,\psi_h^j\big)_\ast=\delta^i_j.
\end{split}
\end{equation}
The standard choices for the  $\ast$-inner product and norm are {the} $L^2(\Omega)$ or $H^1(\Omega)$ ones.  

The solution to \eqref{pod_minim} can be interpreted as finding a subspace $V^{\rm pod}_\ell\subset V_h$, given by $V^{\rm pod}_\ell=\mbox{span} \big\{ \psi_h^1,\dots, \psi_h^\ell\big\}$, that approximates the space spanned by all observed snapshots in the best possible way (subject to the choice of the norm). 
 The Galerkin POD--ROM results from replacing $V_h$ by $V^{\rm pod}_\ell$ in \eqref{eqn:FEM}.

The POD reduced basis captures \emph{cumulative} information regarding the snapshots' dependence on the parameter vector $\balpha$. However, lacking parameter specificity, the basis may lack robustness for parameter values outside the sampling set and may necessitate a high dimension to accurately represent the solution manifold. This limitation poses challenges in the utilization and analysis of POD-based ROMs for parametric problems. To overcome this challenge, a tensor technique based on low-rank tensor decomposition was recently introduced~\cite{mamonov2022interpolatory,mamonov2023tensorial} with the goal of preserving information about parameter dependence in reduced-order spaces.

The LRTD approach can be interpreted as a multi-linear extension of POD. To see this, recall that POD can alternatively be defined as a low-rank approximation of the snapshot matrix. Consider a nodal basis denoted by $\{ \xi_h^i \}_{i=1}^{M}$ in the finite element space $V_h = \text{span}\{\xi_h^1, \dots, \xi_h^M\}$ and define the mass and stiffness finite element matrices $\rM$ and $\rL$, both of size $\mathbb{R}^{M\times M}$, with entries:
\[
\rM_{ij}=(\xi_h^j,\xi_h^i)_0\quad \text{and}\quad \rL_{ij}=(\xi_h^j,\xi_h^i)_1.
\]
Denote by the bold symbol $\bu^j(\widehat\balpha_k)\in\mathbb{R}^{M}$ 
the vector of expansion coefficients of
$u_h^j(\widehat\balpha_k)$ in the nodal basis.  
Consider now the \emph{matrix} of all snapshots
\begin{equation}
\label{eqn:Phi}
\Phi_{\rm pod} = [\bu^1(\widehat\balpha_1), \ldots, \bu^N(\widehat\balpha_1), \ldots,
\bu^1(\widehat\balpha_K), \ldots,\bu^N(\widehat\balpha_K)] \in \R^{M \times N K}
\end{equation}
and let $\hbs^i=\rM^{-\frac12}\bs^i$, where $\{\bs^i\}^\ell_{i=1}$ are the  first $\ell$ 
left singular vectors of $\rM^{\frac12}\Phi_{\rm pod}$. 
Straightforward calculations show that $\hbs^i\in \mathbb{R}^{M}$ are the coefficient vectors of the $L^2$-orthogonal POD basis in $V_h$,
{corresponding to $\ast=L^2$ in \eqref{pod_minim}},
i.e., $\psi_h^i=\sum_j (\hbs^i)_j\xi^j_h$.  The  coefficient vectors of the $H^1$-orthogonal POD basis
{corresponding to $\ast=H^1$ in \eqref{pod_minim}}, 
can be defined through similar calculations with $\rM$ replaced by $\rL$. 

By  the Eckart--Young theorem, the truncated SVD solves the minimization problem of seeking the optimal rank-$\ell$ approximation of a matrix in the spectral and Frobenius norms. When we apply this result within our context, we can interpret the   
POD as finding \emph{the best low-rank approximation} of the  snapshot matrix:
\begin{equation}
	\label{SVDopt}
\widehat\rS \widehat\rV^T =\mbox{arg}\hskip-1ex\min_{ \hskip-3ex {\rA\in\mathbb{R}^{M\times \rev{NK}} \atop  \text{rank}(\rA)\le\ell} }\|\Phi_{\rm pod}-\rA\|_{\ast\ast},\quad\text{with}~ \widehat\rS=[\hbs^1,\dots,\hbs^\ell]\in\mathbb{R}^{M\times\ell}~\text{and}~\widehat\rV^T\in  \mathbb{R}^{\ell\times \rev{NK}},
\end{equation}
with  $\|\cdot\|_{\ast\ast}=  \|\rM^{\frac12}\cdot\|_F$. 


The concept of the LRTD-ROM should now be more accessible to understand. Instead of arranging snapshots in a matrix $\Phi_{\rm pod}$  one seeks to exploit the natural (physical space--time--parameter space)   tensor structure of the snapshots domain and to utilize LRTD instead of matrix SVD for solving a tensor analogue of the low-rank approximation problem~\eqref{SVDopt}.

\begin{remark}\rm 
We  note that  for the quasi-uniform mesh, the matrix $\rM$ is spectrally equivalent to the scaled identity matrix $h^d \rI$,
\begin{equation}\label{Mh}
	c_0  h^d \|\bx\|^2_{\ell^2} \le \bx^T\rM\bx \le c_1  h^d 	\|\bx\|^2_{\ell^2}\quad \forall\, \bx\in\R^M,
\end{equation}
with $c_1,c_0>0$ independent of $h$ ($c_0,c_1$ may depend only on shape regularity of the triangulation and $m$). In this case {and for the purpose of analysis,} $\|\cdot\|_{\ast\ast}$ in \eqref{SVDopt} can be replaced by the Frobenius norm.
\end{remark}


\section{Multi-linear algebra preliminaries and LRTD} 
\label{s:prelim}

Assume that the parameter domain $\cA$ is the $D$-dimensional box 
\begin{equation}
	\cA = {\textstyle \bigotimes\limits_{i=1}^D} [\alpha_i^{\min}, \alpha_i^{\max}].
	\label{eqn:box}
\end{equation}  
Also, assume for a moment that the sampling set $\hcA\subset \cA$ is a Cartesian grid: we distribute $K_i$ nodes $\{\halpha_i^j\}_{j=1,\dots,K_i}$ 
within each of the intervals $[\alpha_i^{\min}, \alpha_i^{\max}]$ in \eqref{eqn:box} for $i=1,\dots,D$, 
and let
\begin{equation}
\label{eqn:grid}
	\hcA = \left\{ \bhalpha =(\halpha_1,\dots,\halpha_D)^T\,:\,
	\halpha_i \in \{\halpha_i^k\}_{k=1}^{K_i}, ~ i = 1,\dots,D \right\},\quad  K= \prod_{i=1}^{D} K_i.
\end{equation} 
 Instead of forming $\Phi_{\rm pod}$ as in \eqref{eqn:Phi},  snapshots are  organized in the \emph{multi-dimensional} array
\begin{equation}
(\bPhi)_{:,j,k_1,\dots,k_D} = \bu^j(\halpha_1^{k_1},\dots,\halpha_D^{k_D}), 
\label{eqn:snapmulti}
\end{equation}
which is a tensor of order $D+2$ and size $M\times N \times K_1\times\dots\times K_D$.  We reserve the first and second indices of $\bPhi$ for dimensions corresponding to the spatial and temporal resolution, respectively.

Unfolding of  $\bPhi$ is reordering  of its elements into a matrix. If all 1st-mode fibers of $\bPhi$, i.e. all vectors $(\bPhi)_{:,j, k_1,\dots,k_D}\in\R^M$,  are organized into columns of a  $M\times NK$ matrix, we get the \emph{1st-mode unfolding matrix}, denoted by $\bPhi_{(1)}$. A particular ordering of the columns in   $\bPhi_{(1)}$ is not important for our purposes. Note that $\Phi_{\rm pod}=\bPhi_{(1)}$ (up to columns permutations).
In the tensor ROM the (truncated) SVD of $\Phi_{\rm pod}$  is replaced with a (truncated) LRTD of $\bPhi$.

{Different notions of tensor ranks exist, and} there is an extensive literature addressing the {problem} of defining tensor rank(s) and LRTD; see e.g.~\cite{hackbusch2012tensor}. 
In \cite{mamonov2022interpolatory,mamonov2023tensorial}, the tensor ROM has been introduced and applied using three widely used LRTD techniques: canonical polyadic, Tucker, and tensor train (TT). These tensor decompositions, in all three low-rank formats, can be seen as extensions of the SVD to multi-dimensional arrays, each with distinct numerical and compression properties. {It is} important to note that the analysis presented in this paper is not contingent on the specific choice of LRTD. However, for the sake of illustration, we will consider the TT-LRTD.

In the TT format ~\cite{TT1}, 
one represents  $\bPhi$ by the following sum of outer products of $D+2$ vectors:
\begin{equation}
	\label{eqn:TTv}
	\bPhi \approx \widetilde{\bPhi} =
	\mbox{\small $\displaystyle \sum_{j_1=1}^{\widetilde R_1}\dots
	\sum_{j_{D+1}=1}^{\widetilde R_{D+1}}$}
	\bt^{j_1}_1 \otimes \bt^{j_1, j_2}_2 \otimes \dots \otimes \bt^{j_D, j_{D+1}}_{{D+1}} \otimes \bt^{j_{D+1}}_{D+2},
\end{equation}
with $\bt^{j_1}_1 \in \R^M$, $\bt^{j_1, j_{2}}_2 \in \R^N$ (space and time directions), and 
{$\bt^{j_i, j_{i+1}}_{i+1} \in \R^{K_{i-1}}$},  $\bt^{j_{D+1}}_{D+2} \in \R^{K_{D}}$ (directions in the parameter space),
{where the outer product is defined entrywise as
\begin{equation}
\begin{split}
\left[ \bt^{j_1}_1 \otimes \bt^{j_1, j_2}_2 \otimes \dots \otimes \bt^{j_D, j_{D+1}}_{D+1} \otimes \bt^{j_{D+1}}_{D+2} \right]_{i_1,i_2,\ldots,i_{D+1},i_{D+2}} = \\
= \left[ \bt^{j_1}_1 \right]_{i_1} \cdot 
\left[ \bt^{j_1, j_2}_2 \right]_{i_2} \cdot 
\ldots \cdot 
\left[ \bt^{j_D, j_{D+1}}_{D+1} \right]_{i_{D+1}} \cdot 
\left[ \bt^{j_{D+1}}_{D+2} \right]_{i_{D+2}},
\end{split}
\end{equation}
for $i_1 = 1,\ldots,M$, $i_2 = 1,\ldots,N$, $\ldots$,
$i_{D+1} = 1,\ldots,K_{D-1}$, $i_{D+2} = 1,\ldots,K_{D}$.}
The positive integers $\widetilde R_i$ are referred to as the {compression ranks} (or TT-ranks). 
{The best approximation of a tensor by a fixed-rank TT tensor always exists and a constructive algorithm  is known} to deliver a quasi-optimal solution~\cite{TT1}. Using this algorithm 
based on the truncated SVD for a sequence of unfolding matrices, one may find $\widetilde{\bPhi}$ in the TT format that satisfies
\begin{equation}
	\label{eqn:TensApproxF}
	\big\| \widetilde{\bPhi} - \bPhi \big\|_F \le  \widetilde\eps \big\|\bPhi \big\|_F
\end{equation}
for a given $\widetilde\eps > 0$. Corresponding TT ranks are then recovered in the course of {the} factorization. 
Here and further, $\|\bPhi \big\|_F$ denotes the tensor Frobenius norm, which is simply the  square root of the sum of the squares of all  entries of $\bPhi$.

By $\Phi_{(xt)}(k_1,\dots,k_D)\in \R^{M\times N}$ we denote a 'space--time' slice of {the} tensor $\bPhi$, 
\[
[\Phi_{(xt)}(k_1,\dots,k_D)]_{mj} =\bPhi_{m,j,k_1,\dots,k_D},
\]
and define another tensor norm
\begin{align*}
\norm{\bPhi}_0&=|\Delta t|^{\frac12}\max_{k_1,\dots,k_D}\|\rM^{\frac12}\Phi_{(xt)}(k_1,\dots,k_D)\|_F{.} \\
\end{align*}
%
With the help of \eqref{Mh}, one readily verifies that
\begin{equation}\label{aux375}
\norm{\bPhi}_0\le (\|\rM\|\Delta t)^{\frac12}\|\bPhi\|_F
\le c_1(h^d\Delta t)^{\frac12}\|\bPhi\|_F.
\end{equation}

For a sufficiently smooth solution to the parametric parabolic problem, the $\norm{\bPhi}_{0}$ norm is bounded in the following sense.   
\begin{lemma}\label{L:Phi_norm}
	Assume the solution of \eqref{eqn:GenericPDE} $u$ is sufficiently regular such that \eqref{UnifEst} is satisfied. 
	Then there holds
	\begin{equation}\label{eqn:normC}
		\norm{\bPhi}_0\le C,
	\end{equation}
	with a constant $C$ independent of $h$, $\Delta t$, and $\hcA$.
\end{lemma} 
\begin{proof} 
	For a fixed $\hbalpha\in\hcA$ denote by	$\bU(\hbalpha) = [\bu^1(\widehat\balpha), \ldots, \bu^N(\widehat\balpha)] \in \R^{M \times N}$ 
	a matrix of snapshots. 	Then by the definition of $\bPhi$ and $\norm{\cdot}_0$ there holds
	$		\norm{\bPhi}_0=|\Delta t|^{\frac12}\max_{\hbalpha\in\hcA}\|\rM^{\frac12}\bU(\hbalpha)\|_F.
	$
	Using the definition of the mass matrix,   \eqref{h_conv} and \eqref{UnifEst} we get
	\begin{equation*}
		\begin{split}
			\Delta t\|\rM^{\frac12}\bU(\hbalpha)\|_F^2& =\Delta t\sum_{j=1}^N\|\rM^{\frac12}\bu^j(\hbalpha)\|_{\ell^2}^2 =\Delta t\sum_{j=1}^N  \|u_h^j(\hbalpha)\|_0^2\\
			&\le 2\Delta t\sum_{j=1}^N  \|u(t^j,\hbalpha)\|_0^2 +2\Delta t\sum_{j=1}^N  \|u(t^j,\hbalpha)-u_h^j(\hbalpha)\|_0^2\\ 
   &{\le 2N \Delta t \max_{j=1,\dots,N}\|u(t^j,\hbalpha)\|_0^2}+ C\,(|\Delta t|^2 + h^{2m})\\
    &{\le 2T\sup_{t\in(0,T)}\|u(t,\hbalpha)\|_0^2} + C\,(|\Delta t|^2 + h^{2m}) \le C
	\end{split}
	\end{equation*}
where $C$ can be taken independent of $h$, $\Delta t$, and $\hbalpha\in\hcA$ .
\end{proof}

Further we assume that a low-rank  approximation of the snapshot tensor $\bPhi$ is available  to satisfy 
\begin{equation}
	\label{eqn:TensApprox}
	\norm{\widetilde{\bPhi} - \bPhi}_0 \le   {\eps}\norm{\bPhi}_0
\end{equation}
with some $\eps>0$. We note that thanks to \eqref{aux375}  this assumption can be practically satisfied with an  a priori given $\eps>0$ by letting in \eqref{eqn:TensApproxF} 
\begin{equation}\label{eps}
 \widetilde\eps = (\|\rM\|\Delta t)^{-\frac12}\norm{\bPhi}_0\|\bPhi\|_F^{-1\MO{\eps}}.
 \end{equation} 

We need a few more notations specific for tensors.
The $k$-mode tensor-vector product $\bPsi \times_k \ba$ of a tensor 
$\bPsi\in \R^{N_1 \times \dots \times N_m}$ of order $m$ and a vector $\ba \in \mathbb{R}^{N_k}$
is a tensor of order $m-1$ and size 
${\small N_1 \times \dots \times N_{k-1} \times N_{k+1} \times \dots \times N_m}$ resulting from the convolution of $\bPsi$ and $\ba$ alone the $k$-th mode. It can be also written as the linear combination of  $N_k$ tensors of size ${\small N_1 \times \dots \times N_{k-1} \times N_{k+1} \times \dots \times N_m}$,
\begin{equation}
	\bPsi \times_k \ba = \sum_{i = 1}^{N_k}a_{i} \bPsi_{:,i,:},
	\label{eqn:kmodeprod}
\end{equation}
where $\bPsi_{:,i,:}\in \R^{N_1 \times \dots \times N_{k-1} \times N_{k+1} \times \dots \times N_m}$  are the $k$-mode slices of $\bPsi$. The representation \eqref{eqn:kmodeprod} and the triangle inequality gives after simple calculation
\[
	\|\bPsi \times_k \ba\|_F \le \|\ba\|_{\ell^1} \max_{i = 1,\dots,N_k} \|\bPsi_{:,i,:}\|_F.
\]

Using the above estimate repeatedly  and the definition of {the} $ \norm{\bPhi}_0$ norm, one checks the following estimate
\begin{equation}\label{aux405}
	\begin{split}
{|\Delta t|^{\frac12}}\|\rM^{\frac12}(\bPhi 
\times_3 \bx^1\times_4 \bx^2 \dots \times_{D+2} \bx^D) \|_F&=
{|\Delta t|^{\frac12}}\|(\rM^{\frac12}\bPhi) 
\times_3 \bx^1\times_4 \bx^2 \dots \times_{D+2} \bx^D \|_F\\
& \le  \norm{\bPhi}_0
\|\bx^1\|_{\ell^1} \|\bx^2\|_{\ell^1} \dots \|\bx^D\|_{\ell^1},
\end{split}
\end{equation}
where $\rM^{\frac12}\bPhi$ denotes a tensor resulting from the multiplication of every 1st-mode fiber by  $\rM^{\frac12}$.


{
\begin{remark}\rm
The accuracy of the low-rank representation in \eqref{eqn:TensApprox} is the sole LRTD bound of importance for the a priori error analysis in this paper. This bound is not specific to any particular format, thus the subsequent analysis remains valid whether $\tbPhi$ is in CP, Tucker, TT, or another low-rank format. Naturally, certain implementation details and compression rates may vary depending on the format, as discussed in \cite{mamonov2022interpolatory}. Here, we employ TT-LRTD as an illustration. Investigating how the recovered compression/tensor ranks are influenced by $\epsilon$, properties of the parametrized PDE solutions, and the tensor format is an important subject that falls outside the scope of this paper. 
\end{remark}
}

\section{LRTD--ROM} 
\label{s:TROM} 
To formulate {the} LRTD--ROM, we need several further notions.
 In the parameter domain we assume an interpolation procedure 
\begin{equation}
	\label{eqn:bea}
	\bchi^i \,:\, \cA \to \mathbb{R}^{K_i},\quad i=1,\dots,D,
\end{equation}
such that for  any {continuous} function $g:\cA\to \R$,
\begin{equation}\label{Interp}
	I(g):= \sum_{k_1=1}^{K_1}\dots \sum_{k_D=1}^{K_D}  \left(\bchi^1 (\balpha ) \right)_{k_1}\dots \left(\bchi^D (\balpha ) \right)_{k_D}
	g(\halpha_1^{k_1},\dots,\halpha_D^{k_D})
\end{equation}
defines an interpolant for
$g$. We assume that the interpolation procedure is stable in the sense that
 \begin{equation}
 	\label{eqn:int_stab}
 	\sum_{j=1}^{K_i} \left| \left(\bchi^i (\balpha ) \right)_j \right| \le C_e,
 \end{equation}
 with some $C_e$ independent of $\balpha\in\cA$ and $i = 1, \ldots, D$.
One straightforward choice is  the Lagrange interpolation of order $p$: for any  $\balpha \in \cA$ let $\widehat{\alpha}_i^{i_1}, \ldots, \widehat{\alpha}_i^{i_p}$ 
be the $p$ closest grid nodes to $\alpha_i$ on $[\alpha_i^{\min}, \alpha_i^{\max}]$, for $i=1,\ldots,D$.
Then 
\begin{equation}
	\label{eqn:lagrange}
	\big(\bchi^i (\balpha)\big)_j = 
	\begin{cases} 
		\prod\limits_{\substack{m = 1, \\ m \neq k}}^{p}(\widehat{\alpha}_i^{i_m}-\alpha_i) \Big/ 
		\prod\limits_{\substack{m = 1, \\ m \neq k}}^{p}(\widehat{\alpha}_i^{i_m}-\widehat{\alpha}_i^j), 
		& \text{if } j = i_k \in \{i_1,\ldots,i_p\}, \\
		\qquad\qquad\qquad\qquad\qquad\qquad\qquad 0, & \text{otherwise}, \end{cases}
\end{equation}
are the entries of $\bchi^i (\balpha)$ for $j=1,\dots,n_i$.

With the help of \eqref{eqn:bea} we introduce the `local' low-rank matrix 
$\widetilde{\Phi} (\balpha)$ via the in-tensor interpolation procedure for tensor $ \widetilde{\bPhi} $:
\begin{equation}
	\label{eqn:extractbt}
	\widetilde{\Phi} (\balpha) = \widetilde{\bPhi} 
	\times_3 \bchi^1(\balpha) \times_4 \bchi^2(\balpha) \dots \times_{D+2} \bchi^D(\balpha) 
	\in \R^{M \times N}.
\end{equation}
If $\balpha = \bhalpha \in \hcA$ belongs to the sampling set, then $\bchi^i(\hbalpha)$ encodes the position of $\widehat{\alpha}_i$ 
among the grid nodes on $[\alpha^{\min}_i, \alpha^{\max}_i]$. Therefore, for $\eps=0$ the  
matrix $\widetilde{\Phi} (\hbalpha)$ is  exactly the matrix of all  
snapshots for the particular $\hbalpha$. For a general $\balpha\in \cA$ the matrix
$\widetilde{\Phi} (\balpha)$ is the result of interpolation between pre-computed snapshots. This interpolation is done directly in a low-rank format.

For an arbitrary given $\balpha \in \cA$  the parameter-specific \emph{local} reduced space $V^\ell(\balpha)$ of dimension $\ell$ is the space spanned by the first  $\ell$ left singular vectors of $\widetilde{\Phi} (\balpha)$:
\begin{equation}
\label{Vrom}V^\ell(\balpha)=\mbox{range}(\rS(\balpha)(1:\ell)),\quad\text{with}~ \{\rS(\balpha),\Sigma(\balpha), \rV(\balpha)\}=\text{SVD}(\widetilde{\Phi} (\balpha)).
\end{equation}
The corresponding {subspace of $V_h$} is denoted by  $V_h^\ell(\balpha)$, i.e., 
\begin{equation}
	\label{Vromh}
V_h^\ell(\balpha)=\{v_h\in V_h\,:\, v_h=\sum_{i=1}^M\xi^i_hv_i,~~\text{for}~(v_1,\dots,v_M)^T\in V^\ell(\balpha)\}.
\end{equation}
Denote by $P_\ell$ the $L^2$-orthogonal projection in $ V_h^\ell(\balpha)$.
The reduced order model then consists in projecting \eqref{eqn:GenericPDE} in  $V_h^\ell(\balpha)$:  For the given $\balpha\in \cA$, find $u_\ell^n\in V_h^\ell(\balpha)$ for $n=1,\dots,N$, solving
\begin{equation}
	\label{eqn:ROM}
	\Big(\frac{u_\ell^{n}-u_\ell^{n-1}}{\Delta t},v_\ell\Big)_0 + a_{\balpha}(u_\ell^n,v_\ell) =\blf_{\balpha}(v_\ell),\quad~   \forall\,v_\ell\in V_h^\ell(\balpha),\quad u^0_\ell=  P_\ell u_0.
\end{equation}

\begin{remark}[Implementation] \label{Rem:Impl}  \rm	
The implementation of the Galerkin LRTD-ROM \eqref{eqn:ROM} follows a two-stage algorithm, as described in \cite{mamonov2022interpolatory, mamonov2023tensorial}.

In the first ({offline}) stage, a low-rank approximation $\widetilde\bPhi$ of the snapshot tensor $\bPhi$ is computed. If {it is} feasible to calculate the full tensor $\bPhi$, standard algorithms like TT-SVD (for TT format), HOSVD (for Tucker format), or ALS (for CP format) can be used to find $\widetilde{\bPhi}$. Otherwise, methods based on low-rank tensor interpolation (e.g., \cite{TT2,ballani2013black,dolgov2019hybrid}) or completion (e.g., \cite{bengua2017efficient,grasedyck2019stable,chen2021auto}) can be applied. These methods work with a sparse sampling of the parameter domain (i.e., the full order model is executed for only a small percentage of $\hcA$). This first stage defines the \emph{universal reduced space} $\widetilde{U}$, which is the span of all 1st-mode fibers of the compressed tensor:
\begin{equation}\label{Univ}
	\widetilde{U} = \text{range}\big(\widetilde\Phi_{(1)}\big).
\end{equation}
For the TT-format, the dimension of $\widetilde{U}$ is equal to $\widetilde R_1$, which is the first compression rank of $\widetilde\bPhi$.  The column vectors $\{\bt^{j}_1\}_{j=1}^{R_1}$ from \eqref{eqn:TTv} form an orthonormal basis in $\widetilde{U}$. {This basis is stored offline.} During the first stage, the system \eqref{eqn:GenericPDE} is projected into the universal space and passed to the second stage, along with the compressed tensor $\widetilde\bPhi$ (in fact, only a part of $\widetilde\bPhi$ is required for online computations).

For any parameter $\balpha\in\cA$, the local ROM space $V^\ell(\balpha)$ is a subspace of $\widetilde{U}$. In the second (online) stage, we find an orthogonal basis in $V^\ell(\balpha)$ by its coordinates in the universal space. The projection of the system into the local space $V^\ell(\balpha)$ and other computations in the online stage are conducted using low-dimensional objects. Moreover, for any new parameter $\bbeta\in\cA$, the coordinates of $V^\ell(\bbeta)$ and the projected system \eqref{eqn:ROM} are recomputed in the online stage using fast (low-dimensional) calculations. {This online stage requires storing projected finite element $\widetilde R_1\times\widetilde R_1$ matrices and a total of $\widetilde R_{D+1}K_D+ \sum^{D}_{i=2}\widetilde R_iK_{i-1}\widetilde R_{i+1}$ real parameters from the TT-LRTD. This \emph{storage} requirement is comparable to that of full POD-ROM, but for new parameter values it enables recomputing parameter-specific systems without relying on offline storage.
}  
For further details on implementation, including a LRTD hyper-reduction technique for nonlinear terms in multi-parameter parabolic problems, refer to \cite{mamonov2022interpolatory,mamonov2023tensorial}.
\end{remark}

\begin{remark}[Interpolation procedure] \rm 
In the TT format, the component-wise interpolation procedure \eqref{eqn:extractbt} simplifies to calculating $\widetilde R_{D+1}+ \sum^{{D}}_{i=2}\widetilde R_i\widetilde R_{i+1}$ inner products between the vectors $\bt_{i>2}^{\cdot,\cdot}$ and sparse vectors $\bchi^j$. Recall that $\widetilde R_{i+1}$ represents the compression ranks. This computation is straightforward to implement and highly computationally efficient. Similar observations hold for  LRTD--ROMs that utilize CP or Tucker LRTD, as discussed in \cite{mamonov2022interpolatory}.

With that said, one can explore various more general interpolation methods in the parameter domain. For instance, {it is} possible to extend \eqref{eqn:bea} to the {tensor-valued} mapping:
\begin{equation*}
	\mathbf{X} \,:\, \cA \to \mathbb{R}^{K_1\times\dots\times K_D}, 
\end{equation*}
and define $\widetilde{\Phi} (\balpha)$ as a convolution of $\tbPhi$ and 
$\mathbf{X}$. 
\end{remark}

\begin{remark}[LRTD vs POD ROMs] \rm	
The universal reduced-order space $\widetilde{U}_h$, which is the FE counterpart of $\widetilde{U}$, shares similarities with the POD space in that both approximate the space spanned by \emph{all} observed snapshots. In fact, if TT-SVD is applied to find $\widetilde\bPhi$ from the complete snapshot tensor, it can be shown that $\widetilde{U}_h = V_\ell^{\rm pod}$ holds true when $\ell=\widetilde R_1$. One key distinction is that in the LRTD-ROM, the universal space is not involved in online computations, and its dimension can be relatively large, corresponding to smaller values of $\varepsilon$ in \eqref{eqn:TensApprox}.

Utilizing LRTD for the snapshot tensor, rather than POD for the snapshot matrix, allows us to determine the  \emph{parameter-specific subspace} $V^\ell_h(\balpha)$ for any incoming parameter $\balpha\in\cA$. Moreover, the dimension of $V^\ell_h(\balpha)$ can be much smaller than the dimension of $\widetilde{U}_h$. This dimension reduction is a key factor in determining the reduced computational complexity of \eqref{eqn:ROM}. For certain parametric problems, Galerkin LRTD-ROMs have shown a dramatic increase in efficiency compared to Galerkin POD-ROMs, as observed in \cite{mamonov2022interpolatory, mamonov2023tensorial}.

Furthermore, we will see that the use of LRTD also enables an error analysis of \eqref{eqn:ROM}, which is not readily available for the Galerkin POD-ROM. This added capability enhances our understanding and control of the error in the reduced-order model.
\end{remark}

\section{Approximation estimate} \label{s:analysis} We are interested in estimating the norm of $\bu-\bu_\ell$, which represents the error of the LRTD--ROM solution. An estimate for the FOM error $\bu-\bu_h$ is provided by \eqref{h_conv}. Therefore, it is sufficient to estimate the norm of $\bu_h-\bu_\ell$. A key result is an approximation estimate for the FOM solution in the local reduced order space. This estimate, discussed in this section, is expressed in terms of the tensor compression accuracy from \eqref{eqn:TensApprox}, characteristics of the mesh in the parameter domain, and the eigenvalues of a snapshot {Gramian} matrix for a \emph{fixed} parameter value. To establish this estimate, we first need some auxiliary results.

\subsection{$\balpha$-uniform estimates for the FOM solution}\label{s:5.1}
We need a certain smoothness of the FOM  solution $u_h$  with respect to parameters. More specifically,  we are interested to show that for an integer $p>0$
there holds $u_h^k(\bx) \in C^p({\cA})$ for all $\bx\in\Omega$ and $k=1,\dots,N$, and 
\begin{equation}
	\label{eqn:Assu}
  	 {\sup_{\balpha\in\cA}\,\sum_{n=1}^N\Delta t \sum_{|\mathbf{j}|\le p}\left\|D^{\mathbf{j}}_\alpha u_h^n\right\|_{L^\infty(\Omega)}^2} \le C_u,	 
\end{equation}
where 
$C_u$ is independent of the discretization parameters,
{$\mathbf{j} = [j_1, j_2, \ldots, j_D]^T \in (\mathbb{Z}^+ \cup \{0\})^D$ with $|\mathbf{j}| = \sum_{i=1}^{D} j_i$, and
\begin{equation}
D^{\mathbf{j}}_\alpha = 
\frac{\partial^{|\mathbf{j}|}}{\partial \alpha_1^{j_1} \partial \alpha_2^{j_2} \ldots \partial \alpha_D^{j_D}}.
\end{equation}}

In this section we show that \eqref{eqn:Assu} holds once the problem data is smooth and the mesh meets certain assumptions. 
We start by recalling a regularity result~\cite[section~IV.9]{ladyzhenskaia1988linear}  for the linear parabolic equation~\eqref{eqn:GenericPDE} with  a general  elliptic  operator $\cL(\balpha)$,
\begin{equation}\label{GenPrab}
	\cL(\balpha)= -\sum_{i,j=1}^{d} a_{ij}(x,t,\balpha)\frac{\partial^2 }{\partial x_ix_j}+\sum_{i=1}^{d} a_{i}(x,t,\balpha)\frac{\partial}{\partial x_i}+  a_0(x,t,\balpha),
\end{equation} 	
and  $\blf_{\balpha} = f(x,t,\balpha)$:
Assume that {the coefficients are $C^\infty(\overline{Q}_T\times\cA)$  functions}, with $\overline{Q}_T=\overline\Omega\times[0,T]$, {the matrix $A(x,t,\balpha)=\{a_{ij}(x,t,\balpha)\}$ is uniformly positive definite in $\overline{Q}_T\times\cA$}, $\partial\Omega$ is smooth, and $f\in W_q^{2\ell-2,\ell-1}(Q_T)$, with $q>1$, and an integer $\ell>0$, $u_0\in  W_q^{2\ell-\frac2q}(\Omega)\cap H^1_0(\Omega)$. 
{Hereafter we denote by $W^{2\ell,\ell}_q(Q_T)$ a Banach space of functions $u\in L_q(Q_T)$ such that all partial derivatives $D^s_xD^r_t u\in L_q( Q_T)$, for $2r+s\le 2\ell$ (see \cite{ladyzhenskaia1988linear} for further details).}
Then for every  $\balpha\in\cA$ the problem \eqref{eqn:GenericPDE} with homogeneous Dirichlet boundary condition, has the unique solution from $W^{2\ell,\ell}_q(Q_T)$, and 
\begin{equation}\label{GenEst}
\|u\|_{W^{2\ell,\ell}_q(Q_T)}\le C\,\left(\|f\|_{W_q^{2\ell-2,\ell-1}(Q_T)}+
\|u_0\|_{W_q^{2\ell-\frac2q}(\Omega)}\right),
\end{equation} 	
{with $C$ independent of $\balpha$.}

 Assume now {the} right hand side and initial condition  in \eqref{GenPrab} {depend} smoothly on {the} parameter $\balpha$, {i.e. $\frac{\partial f}{\partial\alpha_k}\in { W_q^{2\ell-2,\ell-1}(Q_T)}$ and  $\frac{\partial  u_0}{\partial\alpha_k}\in W_q^{2\ell-\frac2q}(\Omega)$} for $k=1,\dots,D$, then a  partial derivative $v=\frac{\partial}{\partial\alpha_k} u$ satisfies 
\begin{equation}\label{GenPrab1}
\begin{split}
	&\frac{\partial v}{\partial t}+\cL(\balpha) v =f_1\quad \text{with}~ f_1=\frac{\partial f}{\partial\alpha_k} - \Big(\frac{\partial}{\partial\alpha_k}\cL(\balpha)\Big) u,\\
	&\text{and}\quad \frac{\partial}{\partial\alpha_k}\cL(\balpha)= -\sum_{i,j=1}^{d} \frac{\partial a_{ij}}{\partial\alpha_k} \frac{\partial^2 }{\partial x_ix_j}+\sum_{i=1}^{d} \frac{\partial a_{i}}{\partial\alpha_k}\frac{\partial }{\partial x_i}+ \frac{\partial a_0}{\partial\alpha_k},
\end{split}
\end{equation} 	 
and $v|_{t=0}=\frac{\partial}{\partial\alpha_k} u_0$. Thanks to the above regularity result and \eqref{GenEst}, the right-hand side in \eqref{GenPrab1} is an element of $W_q^{2\ell-2,\ell-1}(Q_T)$ and there holds 
\begin{multline}\label{GenEst1}
	\left\|\frac{\partial u}{\partial\alpha_k}\right\|_{W^{2\ell,\ell}_q(Q_T)}\le C\,\left(\|f_1\|_{ W_q^{2\ell-2,\ell-1}(Q_T)}+
\left\|\frac{\partial  u_0}{\partial\alpha_k}\right\|_{ W_q^{2\ell-\frac2q}(\Omega)}\right)\\
\le C\,\left(\|f\|_{W_q^{2\ell-2,\ell-1}(Q_T)}+ \|u_0\|_{W_q^{2\ell-\frac2q}(\Omega)}+
\left\|\frac{\partial f}{\partial\alpha_k}\right\|_{ W_q^{2\ell-2,\ell-1}(Q_T)}+ \left\|\frac{\partial  u_0}{\partial\alpha_k}\right\|_{ W_q^{2\ell-\frac2q}(\Omega)}\right) .
\end{multline} 	
Applying the same argument repeatedly, one proves 
 \begin{equation}\label{GenEst2}
 	\begin{split}
 		\|D_\alpha^{\mathbf{j}} u\|_{ W^{2\ell,\ell}_q(Q_T)}&\le C\,\sum_{|\mathbf{j}|\le p'} \left(\|D_\alpha^{\mathbf{j}} f\|_{ W_q^{2\ell-2,\ell-1}(Q_T)}+ \|D_\alpha^{\mathbf{j}} u_0\|_{ W_q^{2\ell-\frac2q}(\Omega)}\right) 
 \end{split}
\end{equation} 	
for all partial derivatives with respect to parameters of order $|\mathbf{j}|\le p'$ with any finite integer $p'$. The constant $C$ in \eqref{GenEst2} depends on $p'$,  H\"older norms of coefficients $a_{i,j},a_i$ in \eqref{GenPrab} and its derivatives with respect to $\balpha$. Therefore, if $f$ and $u_0$ are smooth so that the norms at the right-hand side of \eqref{GenEst2} are uniformly bounded in $\balpha$, then we get
\begin{equation}\label{Aux633}
\sup_{\balpha\in\cA}\sum_{|\mathbf{j}|\le p'}\left\|D^{\mathbf{j}}_\alpha u\right\|_{W^{2\ell,\ell}_q(Q_T)}=C(p')< \infty.
\end{equation} 
For $j=0$, $q=2$, $2\ell\ge \max\{m,4\}$ this estimate implies \eqref{UnifEst}.
If in \eqref{Aux633} we allow $\ell=1$ and $q>(d+2)/2$, then $D^{\mathbf{j}}_\alpha u$ are H\"older-continuous in $x$ and $t$, and in particular there holds
$	\sup\limits_{\balpha\in\cA}\sum\limits_{|\mathbf{j}|\le p'}\left\|D^{\mathbf{j}}_\alpha u\right\|_{L^\infty(Q_T)}< \infty$.

Extending $L^\infty$-estimates to finite element solutions of PDEs  is a technically challenging problem extensively addressed in the literature (see, e.g., \cite[section 6]{thomee2007galerkin} and references therein). However, we are not aware of results for FEM for parametric PDEs, which would imply \eqref{eqn:Assu}. Thus we give below a simple argument, which demonstrates   \eqref{eqn:Assu} for the solutions of \eqref{eqn:FEM} if {$\Delta t \le c\,  |\ln h|^{-\frac12}$ (for $d=2$) or  $\Delta t \le c\, h^{\frac12}$ (for $d=3$) holds, }
the PDE solution satisfies a uniform regularity {estimate} with respect to $\balpha$ as in \eqref{Aux633}, and the dependence of $\cL(\alpha)$ on $\alpha$ is smooth in the following sense: $D_\alpha^{\mathbf{j}} \cL(\balpha)\in \cL(H^1,H^{-1})$ for all $\balpha\in\cA$ and    
\begin{equation}\label{Lsmooth}
\sup_{\balpha\in\cA}\|D_\alpha^{\mathbf{j}} \cL(\balpha)\|_{H^1\to H^{-1}} < \infty\quad \text{for all}~|\mathbf{j}|\le p.
\end{equation}
This assumption is fulfilled by $\cL$ of the general form \eqref{GenPrab} with the coefficients 
{$a_{ij}$, $a_i$ and $a_0$} that 
are {$C^\infty(\overline{Q}_T\times\cA)$}  functions.

First, given the basis $\{\xi_i\}^M_{i=1}$ in $V_h$, the FEM {problem} \eqref{eqn:FEM} can be written as a sequence of linear algebraic systems with matrices and right-hand sides smoothly depending on $\balpha$. Hence the solution also depends smoothly on $\balpha$ and we can differentiate it. Furthermore, note that   $D^{\mathbf{j}}_\alpha u_h^n$ is an element of $V_h$ for any index $\mathbf{j}$. 

Differentiating \eqref{eqn:GenericPDE} and \eqref{eqn:FEM} with respect to a parameter, we find that the difference of  partial derivatives  $\te_h^n= \frac{\partial}{\partial\alpha_k}(u(t_n) -u_h^n)$ satisfies  
\begin{multline}
	\label{eqn:FEM1}
	\Big(\frac{\te_h^{n}-\te_h^{n-1}}{\Delta t},v_h\Big)_0 + a_{\balpha}(\te_h^n,v_h)\\  =\Big(\frac{\partial u_t(t_n)}{\partial\alpha_k} -\frac{\frac{\partial}{\partial\alpha_k} \big(u(t_n)-u(t_{n-1})\big)}{\Delta t},v_h\Big)_0 - \Big\langle\frac{\partial}{\partial\alpha_k}\cL(\balpha)(u(t_n)-u_h^n),v_h\Big\rangle,\quad~   \forall\,v_h\in V_h.
\end{multline}
The first term on the right-hand side of \eqref{eqn:FEM1} is a standard $O(\Delta t)$-consistency term.  To bound the second, we apply the {Cauchy-Schwarz} inequality, \eqref{Lsmooth} and the convergence estimate \eqref{h_conv} to show
\begin{equation*}
\begin{split}
\Delta t\sum_{n=1}^{N}\Big\langle \Big(\frac{\partial}{\partial\alpha_k}&\cL(\balpha)(u(t_n)-u_h^n),v_h^n\Big)\Big\rangle
\le \Delta t\sum_{n=1}^{N}\Big\|\frac{\partial}{\partial\alpha_k}\cL(\balpha)\Big\|_{H^1\to H^{-1}}\|u(t_n)-u_h^n\|_1\|v_h^n\|_1\\
&\le C \Delta t \Big(\sum_{n=1}^{N}\|u(t_n)-u_h^n\|_1^2\Big)^{\frac12}
 \Big(\sum_{n=1}^{N}\|v_h^n\|_1^2\Big)^{\frac12}
 \le C \Big( \Delta t + h^m\Big)
 \Big(\sum_{n=1}^{N}\Delta t\|v_h^n\|_1^2\Big)^{\frac12},
\end{split}
\end{equation*}
with a constant $C$ independent of $\alpha$, $h$ and $\Delta t$. Therefore, {the} same textbook analysis  of proving \eqref{h_conv} gives 
\begin{equation}\label{h_conv1}
	\max_{n=0,\dots,N}\Big\|\frac{\partial}{\partial\alpha_k} (u(t^n)- u^n_h)\Big\|^2_0 +\Delta t \sum_{n=1}^{N}\Big\|\frac{\partial}{\partial\alpha_k}( u(t^n)-u^n_h)\Big\|^2_1\le C\,(\Delta t^2 + h^{2m}),
\end{equation}
for all first order partial derivatives of the FOM error $u(t^n)- u^n_h$. 
{The} constant $C$ depends on Sobolev norms of time and space derivatives of $\frac{\partial}{\partial\alpha_k} u$. It is uniformly bounded in $\balpha$ if 
\begin{equation}\label{UnifEst1}
	\sup_{\balpha\in\cA}\Big(\Big\| \frac{\partial u_0}{\partial\alpha_k} \Big\|_{H^m(\Omega)} + \int_0^T\Big(\Big\|\frac{\partial u_t}{\partial\alpha_k} \Big\|_{H^{{m+1}}(\Omega)}+ \Big\|\frac{\partial u_{tt}}{\partial\alpha_k} \Big\|_0\Big)\,dt  \Big) < \infty
\end{equation}
holds. In turn, to satisfy \eqref{UnifEst1} it is sufficient to assume the corresponding uniform bound for the initial condition and \eqref{Aux633} with $p'=1$, $q=2$ and $2\ell\ge\max\{m+1,4\}$.

The bound in \eqref{h_conv1} resembles the one in \eqref{h_conv}, but for the derivatives. 
Therefore,  using the same argument repeatedly, one shows 
\begin{equation}\label{h_conv2}
	\max_{n=0,\dots,N}\|D_\alpha^{\mathbf{j}} (u(t^n)- u^n_h)\|^2_0 +\Delta t \sum_{n=1}^{N}\|D_\alpha^{\mathbf{j}}( u(t^n)-u^n_h)\|^2_1\le C\,(\Delta t^2 + h^{2m}),
\end{equation}
for all partial derivatives with respect to parameters of order $|\mathbf{j}|\le p$. A constant $C$ is  independent of  $\balpha$, $h$ and $\Delta t$ once \eqref{Aux633} holds with $p'=p$, $q=2$ and $2\ell\ge\max\{m+1,4\}$.

Denote by $P_H$ the $L^2$ projection into the FE space $V_h$. Using the $L^\infty$-stability of $P_H$, which holds for quasi-uniform meshes \cite{douglas1974stability,crouzeix1987stability}, an $L^\infty$ finite element inverse inequality, we get
{
\begin{equation*}
	\begin{split}
\|D_\alpha^{\mathbf{j}} u^n_h\|_{L^\infty(\Omega)} &\le 
\|D_\alpha^{\mathbf{j}} u^n_h-P_H D_\alpha^{\mathbf{j}} u(t^n)\|_{L^\infty(\Omega)}+\|P_H D_\alpha^{\mathbf{j}} u(t^n)\|_{L^\infty(\Omega)}\\
&\le
C\,\big(\ell_h \|D_\alpha^{\mathbf{j}} u^n_h-P_H D_\alpha^{\mathbf{j}} u(t^n)\|_1+\| D_\alpha^{\mathbf{j}} u(t^n)\|_{L^\infty(\Omega)}\big)\\
&\le
C\,\big(\ell_h\|D_\alpha^{\mathbf{j}} (u^n_h- u(t^n))\|_1+ \ell_h\|(1-P_H) D_\alpha^{\mathbf{j}} u(t^n)\|_1+\| D_\alpha^{\mathbf{j}} u(t^n)\|_{L^\infty(\Omega)}\big),
\end{split}
\end{equation*}
with $\ell_h=\left\{\small \begin{array}{ll} |\ln h|^{\frac12}, & d=2 \\  h^{-\frac12}, & d=3\end{array}\right.$ and some $C$ independent of  $\balpha$, $h$ and $\Delta t$. 
Summing up and applying $\|(1-P_H) D_\alpha^{\mathbf{j}} u(t^n)\|_1\le c\,h^m$ and \eqref{h_conv2} gives
\[
\sum_{n=1}^N\Delta t \|D_\alpha^{\mathbf{j}} u^n_h\|_{L^\infty(\Omega)}^2 \le C\,\big(\ell_h^2(\Delta t^2 + h^{2m})+1\big)\le C\,\big(\ell_h^2\Delta t^2 +1\big),
\]
The estimate \eqref{eqn:Assu} follows if $\ell_h\Delta t\le c$.  
}
We proved the following theorem. 

\begin{theorem}\label{Th:param} Consider a family of  quasi-uniform triangulations of $\Omega$. 
Assume that mesh parameters { satisfy $\Delta t \le c\,  |\ln h|^{-\frac12}$ (for $d=2$) or  $\Delta t \le c\, h^{\frac12}$ (for $d=3$)}. Assume the elliptic operator in \eqref{eqn:GenericPDE} has the form \eqref{GenPrab},
	$a_{ij}$, $a_i$, $a_0$ are {$C^\infty(\overline{Q}_T\times\cA)$  functions}, and  right hand side $f$ and initial condition $u_0$ are {sufficiently smooth  functions such that the norms at the right-hand side of \eqref{GenEst2} make sense and are uniformly bounded on $\cA$}. Then \eqref{eqn:Assu} holds with any finite integer $p\ge0$.   
\end{theorem}

{The assumption on $\Delta t$ in the theorem above is not restrictive. Furthermore,   for higher order in time numerical methods, it is further relaxed. For example, for BDF2 implicit method the assumption reads $\Delta t^2 \le c\,  |\ln h|^{-\frac12}$ (for $d=2$) or  $\Delta t^2 \le c\, h^{\frac12}$ (for $d=3$).}

\subsection{An interpolation result}
For the sampling set $\hcA$ in the parameter domain  define the mesh step for the sampling in each  parameter direction:
\begin{equation}
	\delta_i = \max\limits_{1 \leq j  \leq K_i-1} \left| \widehat{\alpha}_j^i - \widehat{\alpha}_{j+1}^i \right|, 
	\quad i = 1,\ldots,D,\quad\text{and let}~ \delta^p = \sum_{i=1}^{D} \delta_i^p,~ p>0.
\end{equation}
We assume that the interpolant defined in \eqref{Interp} is of order $p$ accurate in the following sense: For a {$p$ times continuously differentiable} function $g:\cA\to \R$ it holds
\begin{equation}
	\label{eqn:int_aprox}
	\sup_{\balpha\in\cA} 
	\Big| \big(g - I(g)\big)(\balpha) \Big|\le
	C_a \delta^p \| g \|_{C^p(\cA)}. 
\end{equation}
For the example of linear  interpolation between the nodes in each parameter direction one has $p=2$  and the bounds \eqref{eqn:int_aprox} and \eqref{eqn:int_stab} can be shown to hold with $C_a = \frac18$ and $C_e = 1$.

 For an arbitrary but fixed parameter $\balpha\in\cA$ denote by $\bu^k(\balpha) \in \mathbb{R}^M$ the nodal coefficients vector for the  FOM  finite element solution $u_h^k(\balpha)$ and let 
 {$\rU(\balpha) \in \R^{M \times N}$}
be the corresponding snapshot matrix,
{as defined in the proof of Lemma~\ref{L:Phi_norm}}.
Of course, the snapshots for out-of-sample parameters
may not be practically available, so the matrix $\rU(\balpha)$ serves for the purposes of analysis only. We have the following
result. 
\begin{lemma}\label{L1}
 For the local  matrix $\widetilde{\Phi} (\balpha)$ defined in \eqref{eqn:extractbt} and  $\rU(\balpha)$ defined above, the following estimate holds
	\begin{equation}\label{est_PhiPsi}
	|\Delta t|^{\frac12}\sup_{\balpha\in\cA}\|\rM^{\frac12}\big(\widetilde{\Phi} (\balpha)- \rU(\balpha)\big)\|_F \le C \left(\eps + \delta^p\right),
	\end{equation}
with $C$ independent of $\Delta t$,  $h$, $\delta$ and $\eps$.
\end{lemma}
\begin{proof} 
For a fixed $\balpha\in\cA$ we define an auxiliary matrix  $	\Phi  (\balpha)\in\R^{M\times N}$, which is the interpolation of
space--time slices in the full snapshot tensor: 
	\begin{equation} \label{aux556}
		\Phi  (\balpha) = \bPhi \times_3 \bchi^1 (\balpha) \times_4 \bchi^2(\balpha) \dots \times_{D+2} \bchi^D(\balpha),
	\end{equation}
	and proceed using the triangle inequality
	\begin{equation}
		\label{eqn:aux438}
		\big\|\rM^{\frac12}\big( \bU (\balpha) - \widetilde{\Phi} (\balpha) \big)\big\|_F \le 
		\big\|\rM^{\frac12}\big( \bU (\balpha) - \Phi  (\balpha) \big)\big\|_F + 
		\big\|\rM^{\frac12}\big( \Phi (\balpha) - \widetilde{\Phi}  (\balpha) \big)\big\|_F.
	\end{equation}
Estimates  \eqref{eqn:int_stab}, \eqref{aux405}, and \eqref{eqn:TensApprox} help us to bound the second term
at the right-hand side of \eqref{eqn:aux438},
	\begin{equation}
		\label{eqn:aux442}
		\begin{split}
		{|\Delta t|^{\frac12}} \left\| \rM^{\frac12}\big( \Phi  (\balpha) - \widetilde{\Phi} (\balpha) \big)\right\|_F & =
		{|\Delta t|^{\frac12} }\left\|\rM^{\frac12} (\bPhi - \widetilde{\bPhi}) 
			\times_3 \bchi^1 (\balpha) \times_4 \bchi^2(\balpha) \dots \times_{D+2} \bchi^D(\balpha) \right\|_F \\
			& \le \norm{\bPhi - \widetilde{\bPhi}}_0
			\| \bchi^1 (\balpha) \|_{\ell^1} \| \bchi^2(\balpha) \|_{\ell^1} \dots \| \bchi^D (\balpha) \|_{\ell^1} \\
			& \le (C_e)^D \norm{\bPhi - \widetilde{\bPhi}}_0
			\le (C_e)^D \eps\norm{\bPhi}_0 \le C\eps,
		\end{split}
	\end{equation}
with $C$ independent of $\balpha$, $\eps$ and discretization parameters.

Note the following identity
\[
D^{\bj}_\alpha u_h^n(\balpha,\bx_i)
=\sum_{j=1}^{M}D^{\bj}_\alpha(\bu^n(\balpha))_j\xi_h^j(\bx_i)=D^{\bj}_\alpha(\bu^n(\balpha))_i,
\]
where  we used an observation that $\xi_h^j(\bx_i)=\delta^i_j$ for the Lagrange nodes $\bx_i$ of the spatial triangulation. 
This and  \eqref{eqn:Assu} imply the estimate  
{
\begin{equation}\label{eqn:aux617}
\sum_{n=1}^N\Delta t \sum_{ |\bj|\le p}|D^{\bj}_\alpha (\bu^n(\balpha))_i|^2= 
\sum_{n=1}^N\Delta t\sum_{ |\bj|\le p} |D^{\bj}_\alpha u_h^n(\balpha,\bx_i)|^2\le C_u,
\end{equation}
}
where $C_u$ from  \eqref{eqn:Assu} is independent of $\balpha\in\cA$, and $m$. 

To handle the first term on the left-hand side of \eqref{eqn:aux438}, we  use the identity
\[
\forall\,\balpha\in\cA: \quad
I\big( (\bu^n(\balpha))_i \big)
= \Phi  (\balpha)_{i,n} 
\]
that follows from  \eqref{Interp} and \eqref{aux556}.
We use it together with \eqref{Mh}, \eqref{eqn:int_aprox}, and  \eqref{eqn:aux617} to compute
	\begin{equation}
	\label{eqn:aux450}
\begin{split}
	\Delta t \|\rM^{\frac12}\big(\rU(\balpha)-{\Phi} (\balpha)\big)\|_F ^2&\le 
	c_1\Delta t\, h^d \|\rU(\balpha)-{\Phi} (\balpha)\|_F ^2 \\
& =
 c_1 \Delta t\, h^d \sum_{n=1}^{N} \sum_{i=1}^{M}\big|(\bu^n(\balpha))_i- I\big((\bu^n(\balpha)_i)\big|^2  \\
& \le c_1 \Delta t\, h^d \sum_{n=1}^{N} \sum_{i=1}^{M} C_a^2\delta^{2p}|(\bu^n(\balpha))_i\big|^2_{C^p(\cA)}
 \le {c_1 C_u}  C_a^2 (M h^d)\delta^{2p}.
\end{split}
	\end{equation}
For quasi-uniform triangulations $M h^d\le c$, with $c$ depending only on shape regularity of the triangulation and $m$.
 The result in \eqref{est_PhiPsi} now follows from \eqref{eqn:aux438}, \eqref{eqn:aux442}, and \eqref{eqn:aux450}.
\end{proof}

\subsection{Gramian matrices and singular values} For an arbitrary but fixed parameter $\balpha\in\cA$ the $L^2$-Gramian matrix (sometimes refereed to as $L^2$-correlation matrix) $\rK(\balpha)\in\R^{N\times N}$ is 
\[
\rK(\balpha)=\left(\rK(\balpha)_{i,j}\right),\quad \rK(\balpha)_{i,j}=\frac1N \big(u_h^j(\balpha),u_h^i(\balpha)\,\big)_0.
\] 
From the definition of the mass matrix we have $\big(u_h^j(\balpha),u_h^i(\balpha)\,\big)_0=(\rM\bu^j(\balpha),\bu^i(\balpha))$, which implies $\rK(\balpha)=N^{-1} \rU(\balpha)^T\rM\rU(\balpha)$. Therefore,  the eigenvalues $\lambda_1(\balpha)\ge \lambda_2(\balpha)\ge\dots\ge \lambda_N(\balpha)\ge0$ of $\rK(\balpha)$ and singular values of $\rU(\balpha)$ are related by
\begin{equation}\label{LambdaSigma}
\lambda_i(\balpha)=N^{-1}\sigma_i^2(\rM^{\frac12}\rU(\balpha)).
\end{equation}

We define
	\begin{equation}\label{Lambda}
\Lambda_\ell = \sup_{\balpha\in\cA}\Big(\sum_{i=\ell}^{N}\lambda_i(\balpha)\,\Big).
	\end{equation}
There  clearly holds  $\Lambda_1\ge \Lambda_2\ge\dots\ge \Lambda_N\ge0$, and $\Lambda_1<\infty$ follows from $\sup_{\MO{n}}\sup\limits_{\balpha\in\cA}\|u_h^n(\balpha)\|_0<\infty$. 

\begin{remark}[$\Lambda_\ell$ vs {the singular values of $\Phi_{\rm pod}$}]\rm
A standard analysis of a Galerkin POD--ROM with reduced dimension $\ell$ involves the quantity $\sum\limits_{i>\ell} \sigma_i^2(\Phi_{\rm pod})$ on the right-hand side of an error estimate. {It can be characterized  through the best approximation of observed states by an $\ell$-dimensional subspace of $V_h$ in the following sense: 
\[
h^d\sum\limits_{i>\ell} \sigma_i^2(\Phi_{\rm pod})\simeq \sum\limits_{i>\ell} \sigma_i^2(M^{\frac12}\Phi_{\rm pod})
=\inf_{ \begin{array}{c} 
    \scriptstyle W\subset V_h \\[-0.8ex]
    \scriptstyle \text{dim}(W)=\ell
\end{array}} \sum_{\balpha\in\widehat\cA}\,\sum_{n=1}^N\|u_h^n(\balpha)- \rP_{W}u_h^n(\balpha)\|^2_0,
\]
where $ \rP_{W}$ is the $L^2$-orthogonal projector on $W$.
}
This quantity has a close connection to the Kolmogorov $\ell$-width of the parameterized manifold of solutions,  cf., e.g.,~\cite{bachmayr2017kolmogorov,unger2019kolmogorov}. 
In contrast, $\Lambda_{\ell+1}$ corresponds to the best possible approximation, maximized over 
$\cA$, of a  \emph{single} trajectory  in an $\ell$-dimensional subspace of $V_h$. The best approximation  may vary depending on $\balpha$:
\footnote{For completeness we include the proof. For a fixed $\balpha\in\cA$, we have thanks to \eqref{LambdaSigma} and the Eckart-Young theorem:
\[
N\sum\limits_{i>\ell}\lambda_i(\balpha)=\sum\limits_{i>\ell}\sigma_i^2(\rM^{\frac12}\rU(\balpha))=\inf_{C\in\R^{M\times N}\atop \text{rank}(C)\le\ell}\|\rM^{\frac12}\rU(\balpha)-C\|_F^2\le \inf_{\rW\in\R^{M\times \ell}}\|\rM^{\frac12}\rU(\balpha)-\rM^{\frac12}\rW\rW^T\rM\rU(\balpha)\|_F^2 .
\]
If we restrict the infimum to such $\rW\in\R^{M\times \ell}$ that are $\rM$-orthogonal, i.e. $\rW^T \rM \rW=\rI$, we have 
\[
\sum\limits_{i>\ell}\sigma_i^2(\rM^{\frac12}\rU(\balpha))\le \inf_{\rW\in\R^{M\times \ell} \atop \rW^TM\rW=\rI}\|\rM^{\frac12}\rU(\balpha)-\rM^{\frac12}\rW\rW^T\rM\rU(\balpha)\|_F^2 = \inf_{ \begin{array}{c} 
    \scriptstyle W\subset V_h \\[-0.8ex]
    \scriptstyle \text{dim}(W)=\ell
\end{array}}\sum_{n=1}^N\|u_h^n(\balpha)- \rP_{W}u_h^n(\balpha)\|^2_0.
\]
However, with the particular choice  of $\rW=\rM^{-\frac12}[\bs_1(\rM^{\frac12}\rU(\balpha)),\dots,\bs_\ell(\rM^{\frac12}\rU(\balpha))]$, where $\bs_k(\rM^{\frac12}\rU(\balpha))$ is the $k$-th left singular vector of  $\rM^{\frac12}\rU(\balpha)$, the last inequality becomes the equality which proves the result in \eqref{eq:aux903}.
}

\begin{equation}\label{eq:aux903}
\Lambda_{\ell+1}
=\sup_{\balpha\in\cA} \inf_{ \begin{array}{c} 
    \scriptstyle W\subset V_h \\[-0.8ex]
    \scriptstyle \text{dim}(W)=\ell
\end{array}} \frac1N\sum_{n=1}^N\|u_h^n(\balpha)- \rP_{W}u_h^n(\balpha)\|^2_0.
\end{equation}
If the variation of {the} solution with respect to parameters is strong, one may expect $\Lambda_\ell$ {to decay (much) faster than} $\sum\limits_{i>\ell+1} \sigma_i^2(\Phi_{\rm pod})$.  
{We note that a relevant statistics for characterizing  $\Lambda_\ell$ can be the nonlinear Kolmogorov width introduced in~\cite{temlyakov1998nonlinear}, but we do not pursue this connection in the present paper.}
\end{remark}

Denote by $\tsigma_1(\balpha)\ge \tsigma_2(\balpha)\ge\dots\ge\tsigma_{\hat N}(\balpha)\ge0$, $\hat N=\min\{N,M\}$, the singular values of the local matrix $\widetilde{\Phi} (\balpha)$ defined in \eqref{eqn:extractbt}. 
We have the following 
result.
\begin{lemma}\label{L2}
	The following estimate holds
	\begin{equation}\label{est_singular}
		\Delta t\, h^d \sup_{\balpha\in\cA}\sum_{i=\ell}^{\hat N}\tsigma_i^2(\balpha)\le C\Big(\Lambda_\ell + \eps^2 + \delta^{2p}\Big),\quad \ell=1,\dots,\hat N,
	\end{equation}
with $C$ independent of $\Delta t$, $h$, $\delta$, $\eps$, and $\ell$. 
\end{lemma}
\begin{proof} 
Since the  {Frobenius} norm is unitarily invariant, the result from \cite[Th.5]{mirsky1960symmetric} applied to matrices $\widetilde{\Phi} (\balpha)$ and $\rU(\balpha)$ yields the estimate
\begin{equation}\label{aux705}	
	\|\mbox{diag}\big(\tsigma_1(\balpha)-\sigma_1(\rU(\balpha)),\dots,\tsigma_{\hat N}(\balpha)-\sigma_{\hat N}(\rU(\balpha))\big)\|_F\le \|\widetilde{\Phi} (\balpha)- \rU(\balpha)\|_F.
\end{equation}
Therefore, using $\tfrac{1}2a^2-b^2\le |a-b|^2$  we get for any integer  $1\le \ell\le \hat N$,
\begin{equation}\label{aux710}
	\begin{split}
			\sum_{i=\ell}^{\hat N} \left[\frac12\tsigma_i^2(\balpha)-\sigma_i^2(\rU(\balpha))\right]
			&\le \sum_{i=\ell}^{\hat N} \left[\tsigma_i(\balpha)-\sigma_i(\rU(\balpha))\right]^2 \\
			&\le\|\widetilde{\Phi} (\balpha)- \rU(\balpha)\|^2_F
			 \le 
C\,|\Delta t|^{-1} h^{-d}(\eps + \delta^p)^2,
	\end{split}	
\end{equation}
where the last inequality is implied by \eqref{Mh} and \eqref{est_PhiPsi}. The estimate \eqref{est_singular} follows from \eqref{aux710} 	and the relation \eqref{LambdaSigma} between $\lambda$-s and $\sigma_i^2(\rU(\balpha))$:
\[
\sigma_i^2(\rU(\balpha))\le  \|\rM^{-\frac12}\|^2\sigma_i^2(\rM^{\frac12}\rU(\balpha))
=\|\rM^{-\frac12}\|^2 N \lambda_i(\balpha)
\le c\,h^{-d} |\Delta t|^{-1}\lambda_i(\balpha),
\]
where we used $\sigma_i(AB)\le\|A\|\sigma_i(B)$ inequality for matrix singular values together with \eqref{Mh}  and $\Delta t =T/N$.
\end{proof}


\subsection{Approximation estimate}
Similar to the analysis of {the} POD--ROM for parabolic problems (see, e.g.~\cite{kunisch2001galerkin}), we need
an inverse inequality in a snapshot space. For a fixed $\balpha\in\cA$, define the space 
\[
V_h(\balpha)=\mbox{span}\{u_h^1(\balpha),\dots,u_h^N(\balpha)\}
\]  
and let $c_{\rm inv}(\balpha)$ be the best constant from the inverse inequality
\begin{equation}\label{Inv}
\|v_h\|_1\le c_{\rm inv}(\balpha)\|v_h\|_0\quad \forall\,v_h\in V_h(\balpha)+\widetilde{U}_h,
\end{equation}
where $\widetilde{U}_h$ is the FE counterpart of the (low-dimensional) universal space in \eqref{Univ}. We define
\[
C_{\rm inv}=\sup_{\balpha\in\cA}c_{\rm inv}(\balpha).
\]
The FE inverse inequality implies $C_{\rm inv}\le c\, h^{-1}$, with $c>0$ independent of $h$. However, the $O(h^{-1})$ bound on $C_{\rm inv}$ can be pessimistic. 

Recall {that $\rP_\ell$ denotes} the $L^2$-orthogonal projection from $V_h$ into the local space $V_h^\ell(\balpha)$. 
The following theorem is the main result of this section. 

\begin{theorem}
	\label{Th1}
	Let  $u_h(t,\balpha)$ be the FOM solution of the finite element problem \eqref{eqn:FEM}  for  $\balpha \in \cA$. The following uniform in $\balpha$ approximation {estimate holds}
	\begin{align}
		\label{eqn:aprox1}
		\sup_{\balpha\in\cA}\Delta t \sum_{n=1}^{N} 
		\left\| u_h^n( \balpha) -  \rP_\ell u_h^n(\balpha) \right\|^2_{1} 
		\le  \tilde C\,C_{\rm inv}\Big(\eps^2 + \delta^{2p}+ \Lambda_{\ell+1}\Big) , 
	\end{align}
with $\tilde C$ independent of $h$, $\Delta t$, $\delta$, $\ell$, and $\eps$.
\end{theorem}

\begin{proof}
Consider the SVD of $\widetilde{\Phi}  (\balpha) \in \R^{M \times N}$ given by
\begin{equation}
	\widetilde{\Phi} (\balpha) = {\rS}(\balpha) {\Sigma}(\balpha) {\rV}^T(\balpha),
	~~\text{with}~~ {\Sigma}(\balpha) = \text{diag} (\tsigma_1,\dots,\tsigma_N).
\end{equation}
Then the first $\ell$ columns of ${\rS}(\balpha)$ form an orthogonal basis in $V^\ell(\balpha)$, cf. \eqref{Vrom}. Consider a matrix formed by these columns $\rS _\ell = \rS(\balpha)(1:\ell) \in \R^{M \times \ell}$. The matrix representation of the orthogonal projection is $\rI  - \rS _\ell \rS _\ell^T$. By the definition of the mass matrix and the projection, we get  
\begin{equation}\label{aux666}
\left\| u_h^n( \balpha) -  \rP_\ell u_h^n(\balpha) \right\|^2_{0}=
\left\| \rM^{\frac12}(\rI  - \rS _\ell \rS _\ell^T)\bu^n(\balpha)\right\|^2_{\ell^2}.
\end{equation}
Summing up over $n$ and using the definitions of the Frobenius  matrix norm, we get
\begin{equation}\label{aux672}
\sum_{n=1}^{N} 
\left\| u_h^n( \balpha) -  \rP_\ell u_h^n(\balpha) \right\|^2_{0}
=
\left\|\rM^{\frac12} (\rI  - \rS _\ell \rS _\ell^T) \rU(\balpha)\right\|^2_{F}.
\end{equation}
We treat the term at the right-hand {side} of \eqref{aux672} as follows:
\begin{align}
	\left\|\rM^{\frac12} (\rI  - \rS _\ell \rS _\ell^T) \rU(\balpha) \right\|^2_F 
		& \le  \left( \left\| \rM^{\frac12}(\rI - \rS _\ell \rS _\ell^T) (\rU(\balpha) - \widetilde{\Phi} (\balpha)) \right\|_F
	+ \left\| \rM^{\frac12}(\rI - \rS _\ell \rS _\ell^T) \widetilde{\Phi} (\balpha) \right\|_F \right)^2 \nonumber \\
	& \le c\, h^{d} \left( \left\| ({\bU}(\balpha) - \widetilde{\Phi} (\balpha)) \right \|_F
	+ \left\| (\rI - \rS _\ell \rS _\ell^T) \widetilde{\Phi} (\balpha) \right\| \right)^2, \label{eqn:aux413}
\end{align}
where we used {the} triangle inequality, \eqref{Mh}, and $\|\rI - \rS _\ell \rS _\ell^T\| \le 1$ for the spectral norm of the projector.
For the last term in \eqref{eqn:aux413}, we observe
\begin{equation}\label{aux687}
	\left\| (\rI - \rS _\ell \rS _\ell^T) \widetilde{\Phi} (\balpha) \right\|_F =
	\left\|{\rS}(\balpha)\, \text{diag}(0,\dots,0,\tsigma_{\ell+1},\dots,\tsigma_N) \; {\rV}^T \right\|_F \le
	\left(  \sum_{j = \ell+1}^{N} \tsigma_j^2 \right)^{\frac12}.
\end{equation}
Combining \eqref{aux666}--\eqref{aux687} together with the results of \eqref{Mh}, Lemmas~\ref{L1} and~\ref{L2} proves the theorem.\\
\end{proof}

\section{Error estimate}\label{s:error}  For the error analysis we apply a standard argument. The error between ROM and FOM solutions is split into two parts
\[
u_h-u_\ell= (u_h- \rP_\ell u_h)+ (\rP_\ell u_h-u_\ell) =: e_h + e_\ell,
\]
which satisfy 
\begin{equation}
	\label{eqn:ERR}
	\Big(\frac{e_\ell^{n}-e_\ell^{n-1}}{\Delta t},v_\ell\Big)_0 + a_{\balpha}(e_\ell^n,v_\ell) =-a_{\balpha}(e_h^n,v_\ell),\quad~   \forall\,v_\ell\in V_h^\ell(\balpha),
\end{equation}
for $n=1,\dots,N$.  The equation \eqref{eqn:ERR} results from subtracting \eqref{eqn:ROM} from \eqref{eqn:FEM} tested with $v_h=v_\ell$. 
We also used that 
\[
\Big(\frac{e_h^{n}-e_h^{n-1}}{\Delta t},v_\ell\Big)_0=0,
\]
since $\rP_\ell u_h$ is the $L^2$-projection of $u_h$. 

Letting $v_\ell=e_\ell^n$ gives after multiplying by  $2\Delta t$
\begin{equation*}
	\|e_\ell^{n}\|^2_0 -\|e_\ell^{n-1}\|^2_0 + 	\|e_\ell^{n}- e_\ell^{n-1}\|^2_0 + 2\Delta t a_{\balpha}(e_\ell^n,e_\ell^n) =-2\Delta t a_{\balpha}(e_h^n,e_\ell^n).
\end{equation*}
We now apply ellipticity and continuity estimates from   \eqref{a_cond} to {$a_{\balpha}$-forms} to obtain 
\begin{equation*}
	\|e_\ell^{n}\|^2_0 -\|e_\ell^{n-1}\|^2_0 + 	\|e_\ell^{n}- e_\ell^{n-1}\|^2_0 + 2\Delta t c_a \|e_\ell^n\|^2_1 \le 2\Delta t C_a\|e_h^n\|_1 \|e_\ell^n\|_1\le \Delta t c_a^{-1}C_a^2\|e_h^n\|_1^2+ c_a\Delta t\|e_\ell^n\|_1^2.
\end{equation*}
The summation over $n$ and elementary computations give
\begin{equation*}
	\|e_\ell^{N}\|^2_0 + \Delta t c_a \sum_{n=1}^N\|e_\ell^n\|^2_1 \le \Delta t c_a^{-1}C_a^2 \sum_{n=1}^N \|e_h^n\|_1^2,
\end{equation*}
where we also used $e_\ell^{0}=0$. With the help of triangle inequality and the approximation result from Theorem~\ref{Th1} we get
\begin{equation} \label{aux930}
	\begin{split}
\Delta t \sum_{n=1}^N\|u_h^n-u_\ell^n\|^2_1 &\le C\, \Delta t \sum_{n=1}^N \|e_h^n\|_1^2
\le  C\, C_{\rm inv}\big( \eps^2 + \delta^{2p} + \Lambda_{\ell+1}\big),
	\end{split}
\end{equation}
	with some $C$ independent of $h$, $\Delta t$, $\delta$, $\ell$, and $\eps$.

Applying the triangle inequality   together with \eqref{aux930} and \eqref{h_conv}  proves our main convergence result 
\begin{theorem}
	\label{Th2}
	Let  $u=u(x,t,\balpha)$ be a solution to \eqref{eqn:GenericPDE} {such that \eqref{UnifEst} is satisfied and let}  $u_\ell$ be the {LRTD-ROM} solution to \eqref{eqn:ROM}. The following error estimate holds
	\begin{align}
				\label{eqn:err1}
		\sup_{\balpha\in\cA}\Delta t \sum_{n=1}^{N} 
		\left\| u(t_n) -  u_\ell^n \right\|^2_{1} 
		\le C\left ( \Delta t^2 + h^{2m} + C_{\rm inv}\big( \eps^2 + \delta^{2p} + \Lambda_{\ell+1}\big) \right) , 
	\end{align}
	with some $C$ independent of $h$, $\Delta t$, $\delta$, $\ell$, and $\eps$.
\end{theorem}

\begin{remark}
{
    \rm  Parameter $\eps$ in \eqref{eqn:err1} measures the relative LRTD  accuracy in the $\norm{\cdot}_0$ tensor norm, which works well for the  error estimates uniform over parameter domain. 
    Alternatively, eq. \eqref{aux375} allows to bound  $\norm{\bPhi - \widetilde{\bPhi}}_0$ in \eqref{eqn:aux442} with 
    $(\Delta t h^d)^{\frac12}\norm{\bPhi - \widetilde{\bPhi}}_F$. In turn, the latter is less than $\widetilde \eps (\Delta t h^d)^{\frac12}\norm{\bPhi}_F$ from \eqref{eqn:TensApproxF}, which would  replace $\eps$ in \eqref{eqn:err1}.
    While this is an overestimate, the quantities $\widetilde \eps$ and $\norm{\bPhi}_F$ are commonly  accessible as an output of  standard LRTD compression algorithms. 
    Moreover, if HOSVD or TT-SVD are used to find $\widetilde{\bPhi}$, then $\norm{\bPhi - \widetilde{\bPhi}}_F$ enjoys upper bounds in terms of the tails of singular values of certain matrix  unfoldings of the tensor $\bPhi$~\cite{de2000multilinear,TT1}.
}
\end{remark}
\section{Numerical examples} 
\label{s:num} 

We assess the bound \eqref{eqn:err1} numerically in two examples from \cite{mamonov2022interpolatory}.
Overall, the results are in good agreement with the conclusions of Theorem~\ref{Th2}.

\subsection{Heat equation}
\label{sec:heat}

The first example is a dynamical system corresponding to a heat equation
{\begin{equation}
u_t(\bx, t, \balpha) = \Delta u(\bx, t, \balpha), \quad \bx \in \Omega, \quad t \in (0, T)
\label{eqn:heat}
\end{equation}}
in a rectangular domain with three holes 
$\Omega = [0,10]\times[0,4] \setminus (\Omega_1 \cup \Omega_2 \cup \Omega_3)$ 
shown in Figure~\ref{fig:dom3hole}. Zero initial condition is enforced and the terminal time is set to $T = 20$.
The system has $D = 2$ parameters that enter the boundary conditions:
\begin{eqnarray}
\left. (\bn \cdot \nabla u + \alpha_1(u-1)\,) \right|_{\Gamma_o} & = & 0,
\label{eqn:bco} \\
\left. \left( \bn \cdot \nabla u + \frac{1}{2} u \right) \right|_{\partial \Omega_j} & =&  \frac{1}{2} \alpha_{2},
\quad j=1,2,3, \label{eqn:bcj} \\
\left. (\bn \cdot \nabla w) \right|_{\partial ([0,10]\times[0,4]) \setminus \Gamma_o} & = & 0,
\label{eqn:bcins}
\end{eqnarray}
where, $\bn$ is the outer unit normal. 
The parameter domain is the box $\cA = [0.01, 0.501] \times [0, 0.9]$.
The system \eqref{eqn:heat}--\eqref{eqn:bcins} is discretized with $P_2$ finite elements on a quasi-uniform 
triangulation of $\Omega$ with maximum element size $h$.

\begin{figure}[ht]
\begin{center}
\includegraphics[width=0.5\textwidth]{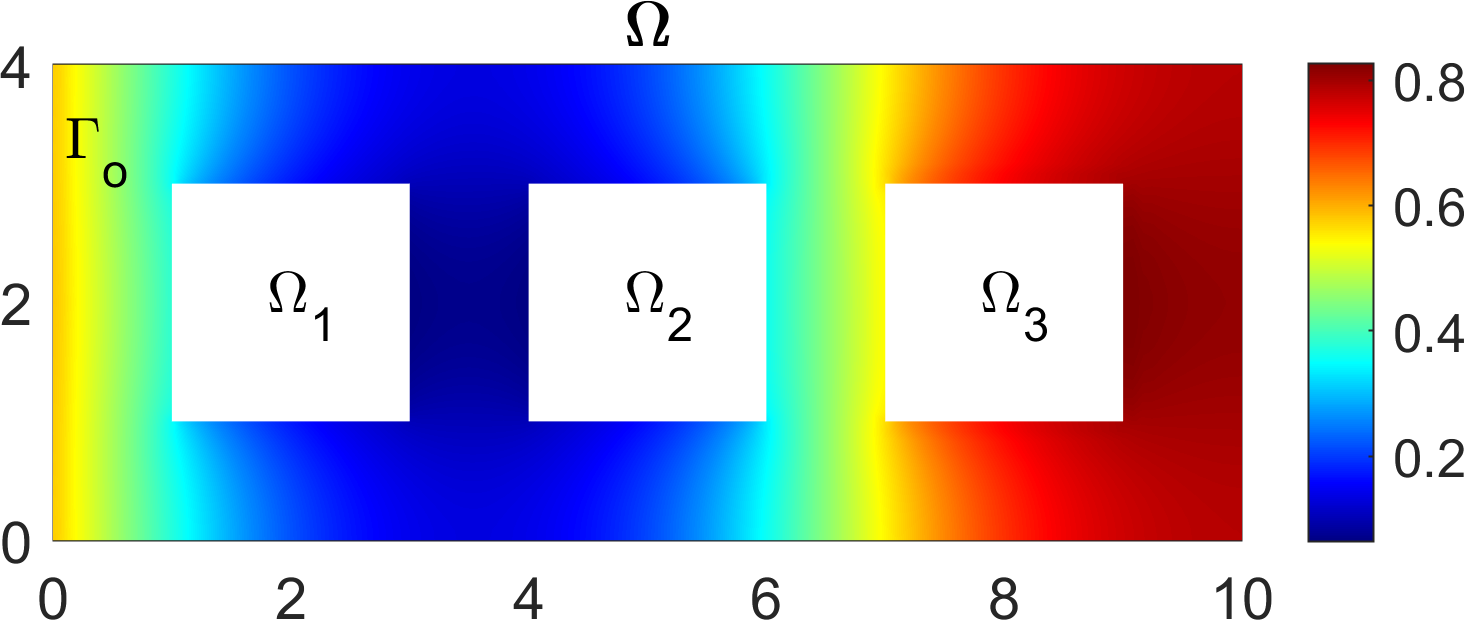} 
\end{center}
\caption{Domain $\Omega$ and the solution $u(T, \bx)$ of 
\eqref{eqn:heat}--\eqref{eqn:bcins} corresponding to $\balpha = (0.5, 0.9)^T$.}
\label{fig:dom3hole}
\end{figure}

\begin{figure}[ht]
\begin{center}
\begin{tabular}{ccc}
\includegraphics[width=0.32\textwidth]{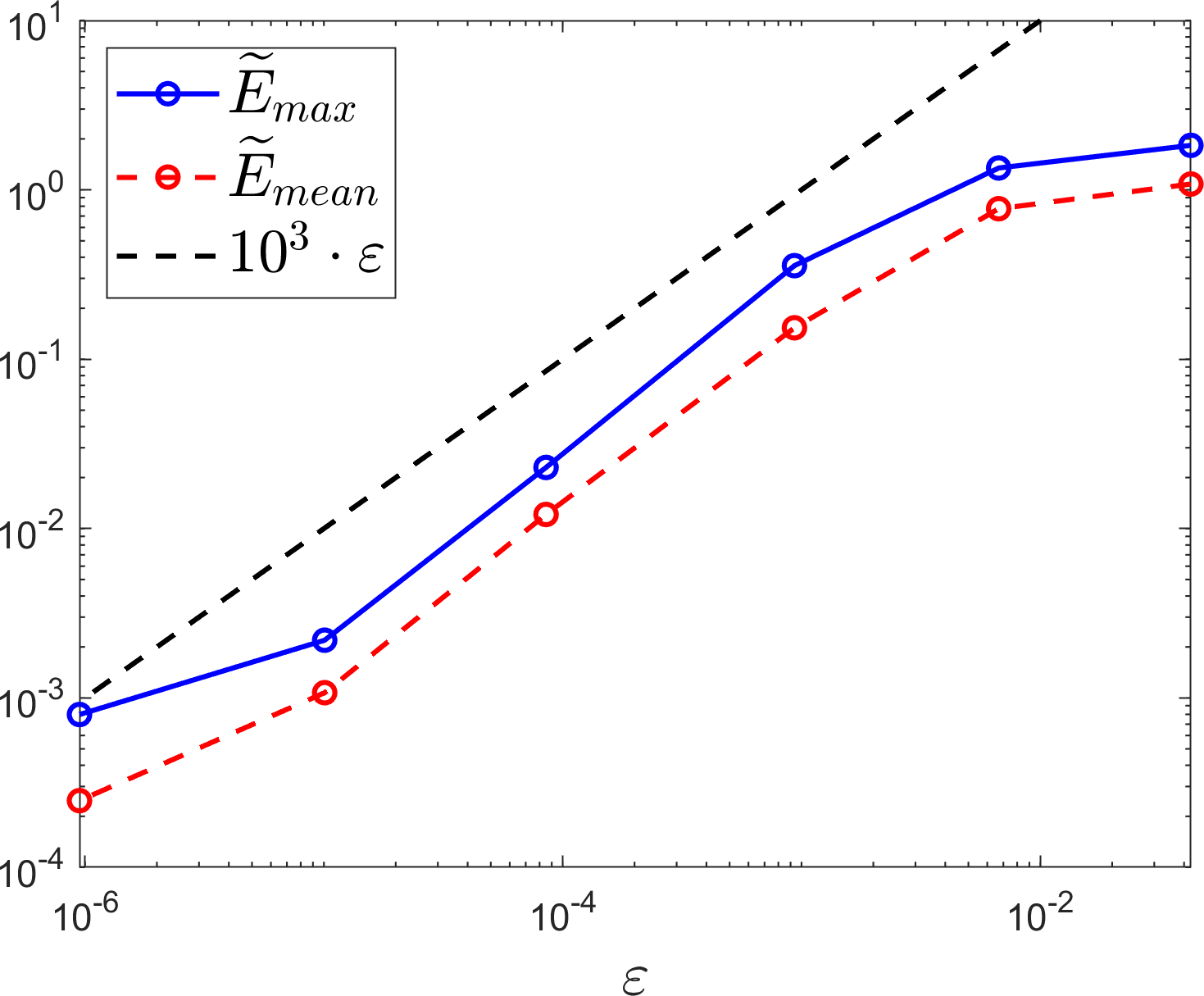} &
\includegraphics[width=0.32\textwidth]{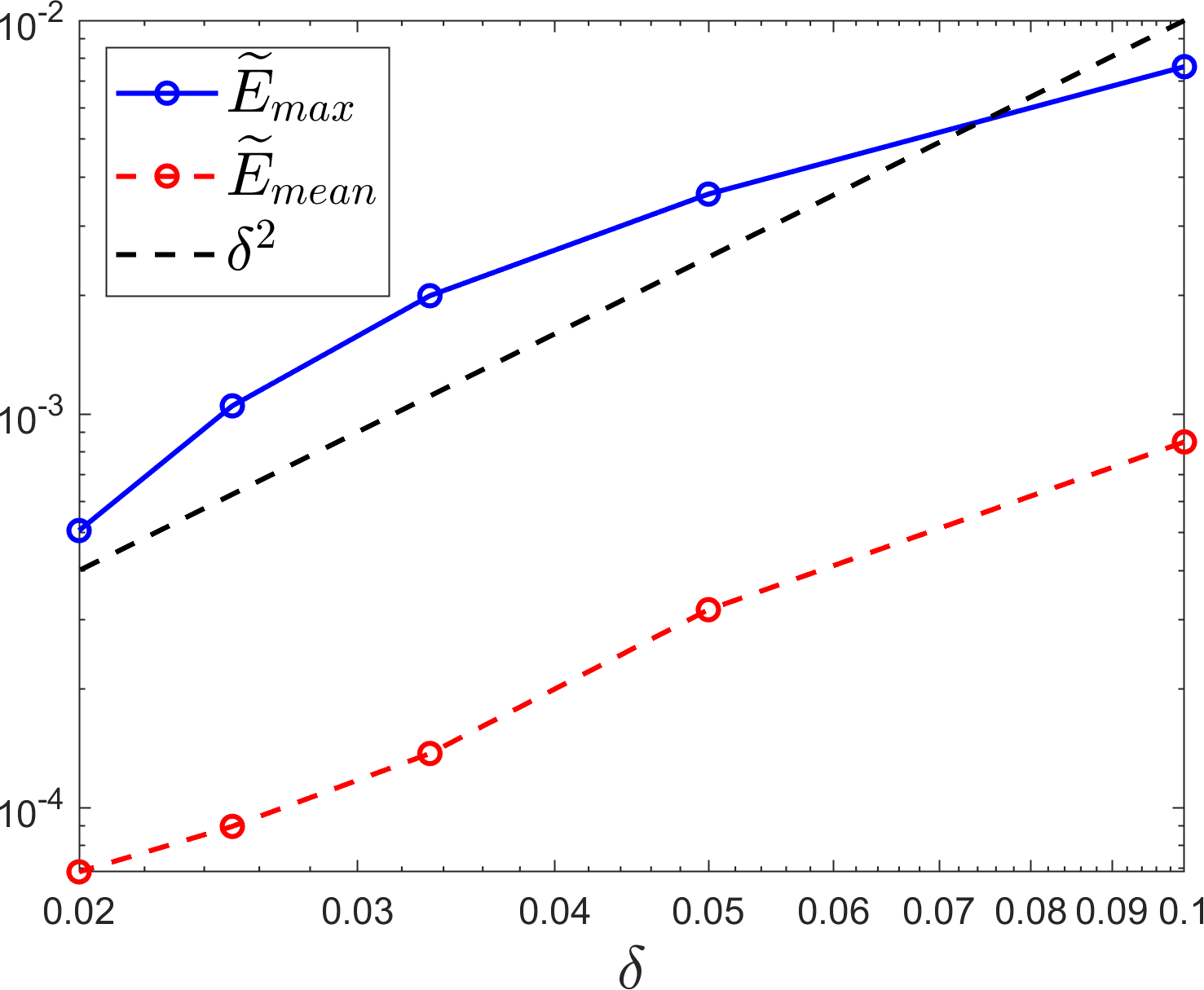} &
\includegraphics[width=0.32\textwidth]{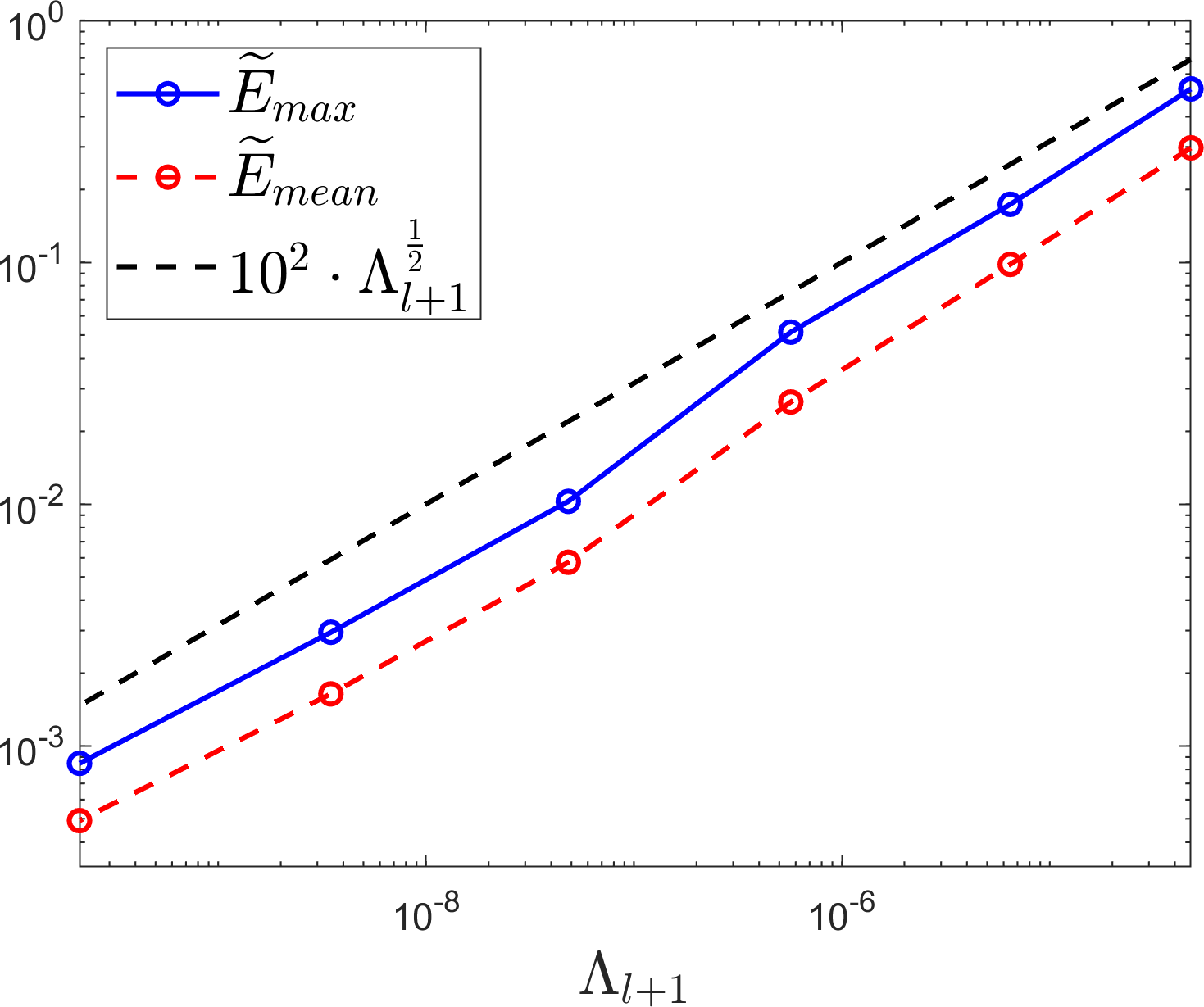} 
\end{tabular}
\end{center}
\caption{Error estimates $\widetilde{E}_{\max}$ (solid blue line with circles) and $\widetilde{E}_{\text{mean}}$ 
(dashed red line with circles) as functions of $\eps$ (left), $\delta$ (middle) and $\Lambda_{l+1}$ (right).
The slopes of dashed black lines represent the theoretical bounds from the right-hand side of \eqref{eqn:err1}:
$\eps$ (left), $\delta^2$ (middle) and $\Lambda_{l+1}^{1/2}$ (right).}
\label{fig:err3hole}
\end{figure}

Denote the error squared from \eqref{eqn:err1} for a single parameter value
\begin{equation}
E_{\balpha} = \Delta t \sum_{n=1}^{N} \left\| u(t_n) -  u_\ell^n \right\|^2_{1},
\label{eqn:erra}
\end{equation}
where $\Delta t = 0.2$ and $N = 100$.
To estimate the error on the left-hand side of \eqref{eqn:err1} we compute numerically 
\begin{eqnarray}
\widetilde{E}_{\max} & = & 
\left( \max_{\balpha\in \widetilde{\cA}} E_{\balpha} \right)^{1/2}, 
\label{eqn:errestmax} \\
\widetilde{E}_{\text{mean}} & = & 
\left( \frac{1}{\widetilde{K}} \sum_{\balpha \in \widetilde{\cA}} E_{\balpha} \right)^{1/2},
\label{eqn:errestmean}
\end{eqnarray}
where $\widetilde{\cA}$ is the  $32 \times 32$ uniform Cartesian grid  in the parameter domain and $\widetilde{K}=1024$.
Note that  $\widetilde{\cA}$ will be always different than the sampling (training) set $\widehat{\cA}$ defined in \eqref{eqn:grid}. 

Obviously, the error \eqref{eqn:erra} and therefore \eqref{eqn:errestmax}--\eqref{eqn:errestmean}
depend on the quantities on the right-hand side of the estimate \eqref{eqn:err1}: 
$\Delta t$, $h$, $\eps$, $\delta$, $p$, and $\Lambda_{l+1}$. Here we fix $p=2$, $\Delta t = 0.2$ and $h = 0.1$
and display in Figure~\ref{fig:err3hole} error estimates $\widetilde{E}_{\max}$ and $\widetilde{E}_{\text{mean}}$
for $\eps$, $\delta$ and $\Lambda_{l+1}$ that vary in the following ranges:
$\eps \in [10^{-6}, 10^{-1}]$, $\delta \in [0.02, 0.1]$ and 
$\Lambda_{l+1} \in [2 \cdot 10^{-10}, 5 \cdot 10^{-5}]$. We observe in Figure~\ref{fig:err3hole} 
that error estimates $\widetilde{E}_{\max}$ and $\widetilde{E}_{\text{mean}}$ behave exactly as predicted 
by \eqref{eqn:err1}, i.e., their rate of decay with respect to $\eps$, $\delta$ and $\Lambda_{l+1}$
is $\eps$, $\delta^2$ and $\Lambda_{l+1}^{1/2}$, respectively. As the exact solution is not explicitly known, conducting a convergence study in terms of $h$ and $\Delta t$ is not feasible. We observed that varying $h$ while keeping other parameters fixed has minimal impact on $\widetilde{E}_{\max}$ and $\widetilde{E}_{\rm mean}$ --- that is, on the error between the ROM and FOM solutions. 
 
\subsection{Advection-diffusion equation}
The second numerical example involves more ($D=5$) parameters than that in Section~\ref{sec:heat}
and corresponds to the dynamical system resulting from the discretization of a linear advection-diffusion equation
{\begin{equation}
u_t(\bx, t, \balpha) = \nu \Delta u(\bx, t, \balpha) - \bseta (\bx, \balpha) \cdot \nabla u(\bx, t, \balpha) + f(\bx),
 \quad \bx \in \Omega, \quad t \in (0, T)
\label{eqn:advdiff}
\end{equation}}
in the unit square domain $\Omega = [0,1] \times [0,1] \subset \R^2$, $\bx = (x_1, x_2)^T \in \Omega$,
with terminal time $T=1$.
Here $\nu = 1/30$ is the diffusion coefficient, $\bseta: \Omega \times \cA \to \R^2$ is the parameterized  advection field 
and $f(\bx)$ is a Gaussian source
\begin{equation}
f(\bx) = \frac{1}{2 \pi \sigma_s^2} \exp \left( - \cfrac{ (x_1 - x_1^s)^2 + (x_2 - x_2^s)^2}{2 \sigma_s^2} \right),
\label{eqn:advdiffsrc}
\end{equation}
with $\sigma_s = 0.05$, $x_1^s = \rev{x_2^s} = 0.25$. 
Homogeneous Neumann boundary conditions and zero initial condition are imposed
\begin{equation}
\left. \left( \bn \cdot \nabla u \right) \right|_{\partial \Omega} = 0, \quad u(\bx, t, \balpha) = 0.
\label{eqn:advdiffbcic}
\end{equation}
The model is parametrized with $D = 5$ parameters with the divergence free advection field $\bseta$ depending on 
$\balpha \in \R^5$. The advection field is defined as follows
\begin{equation}
\bseta(\bx, \balpha) = \begin{pmatrix} \eta_1(\bx, \balpha) \\ \eta_2(\bx, \balpha) \end{pmatrix}  = 
\begin{pmatrix} \cos(\pi/4) \\ \sin(\pi/4) \end{pmatrix} 
+ \frac{1}{\pi} \begin{pmatrix} \partial_{x_2} h(\bx, \balpha) \\ - \partial_{x_1} h(\bx, \balpha) \end{pmatrix},
\label{eqn:etacos}
\end{equation}
where $h(\bx)$ is the cosine trigonometric polynomial
\begin{equation}
\begin{split}
h(\bx, \balpha) = & \;\;\;\; \alpha_1 \cos(\pi x_1) + \alpha_2 \cos(\pi x_2) + \alpha_3 \cos(\pi x_1) \cos(\pi x_2) \\
& + \alpha_4 \cos(2\pi x_1) + \alpha_5 \cos(2\pi x_2).
\end{split}
\end{equation}
The parameter domain is a 5D box $\mathcal{A} = [- 0.1, 0.1]^5$. The system 
\eqref{eqn:advdiff}--\eqref{eqn:advdiffbcic} is discretized similarly to \eqref{eqn:heat}--\eqref{eqn:bcins}.

\begin{figure}[ht]
\begin{center}
\begin{tabular}{ccc}
\includegraphics[width=0.32\textwidth]{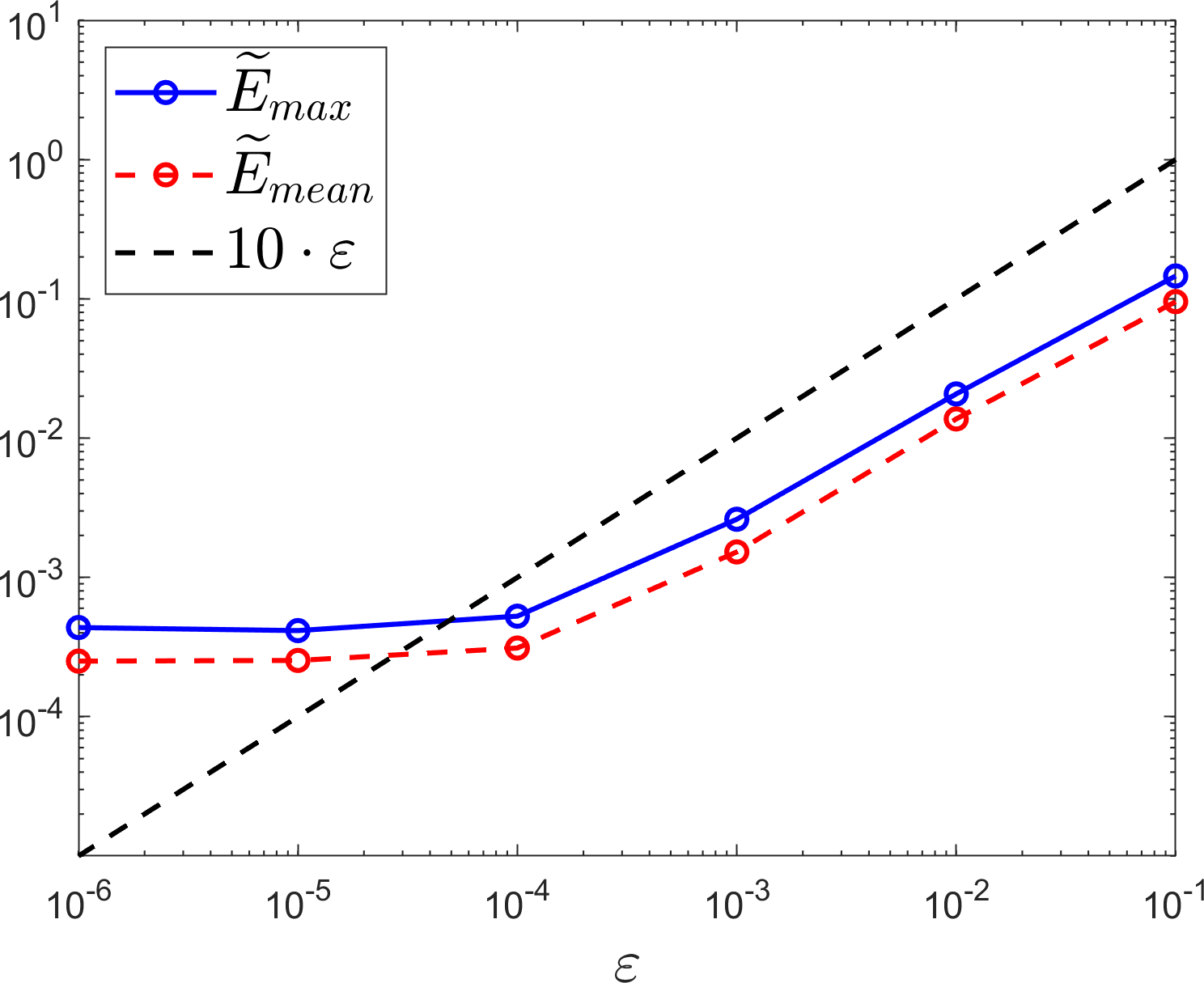} &
\includegraphics[width=0.32\textwidth]{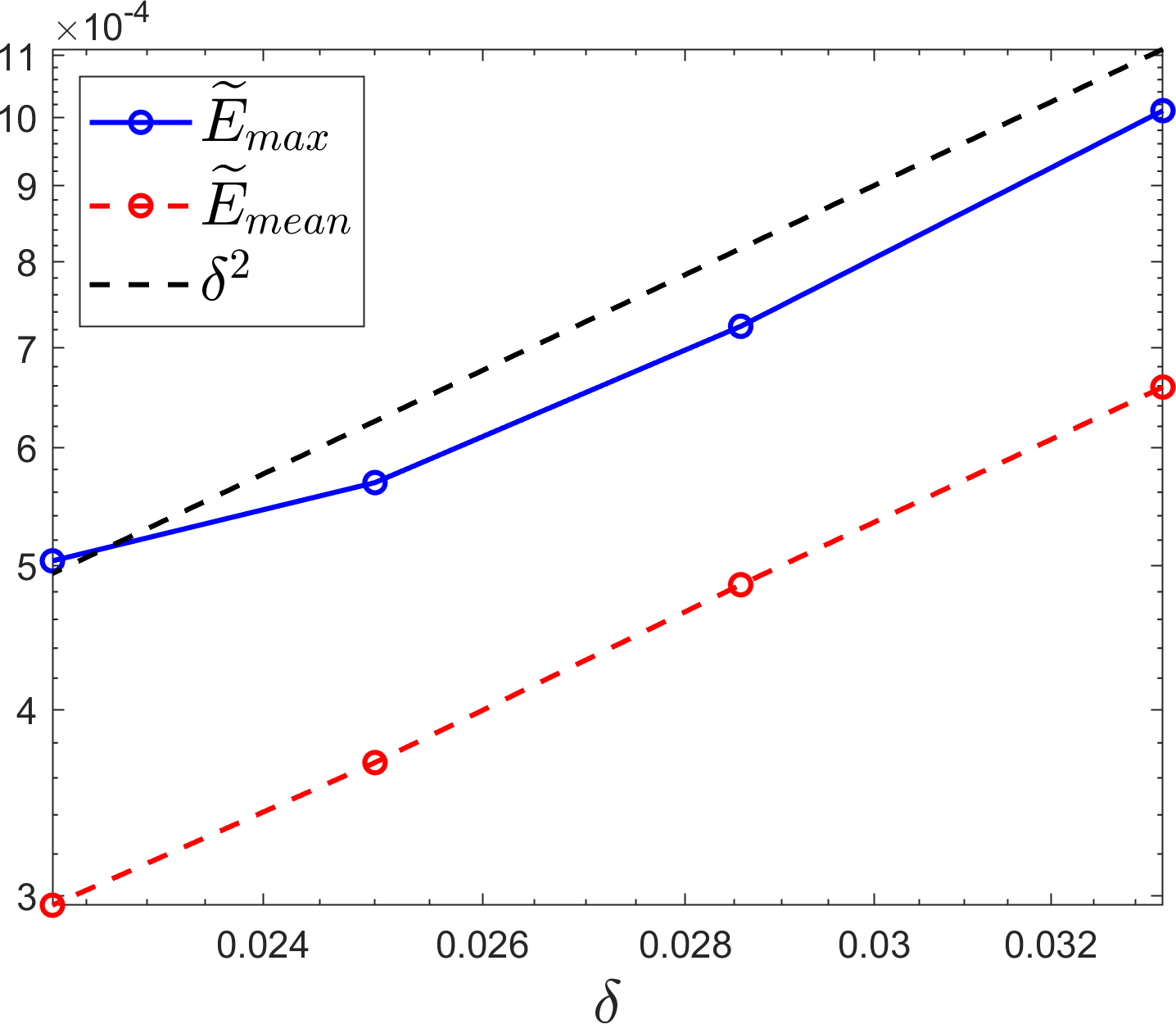} &
\includegraphics[width=0.32\textwidth]{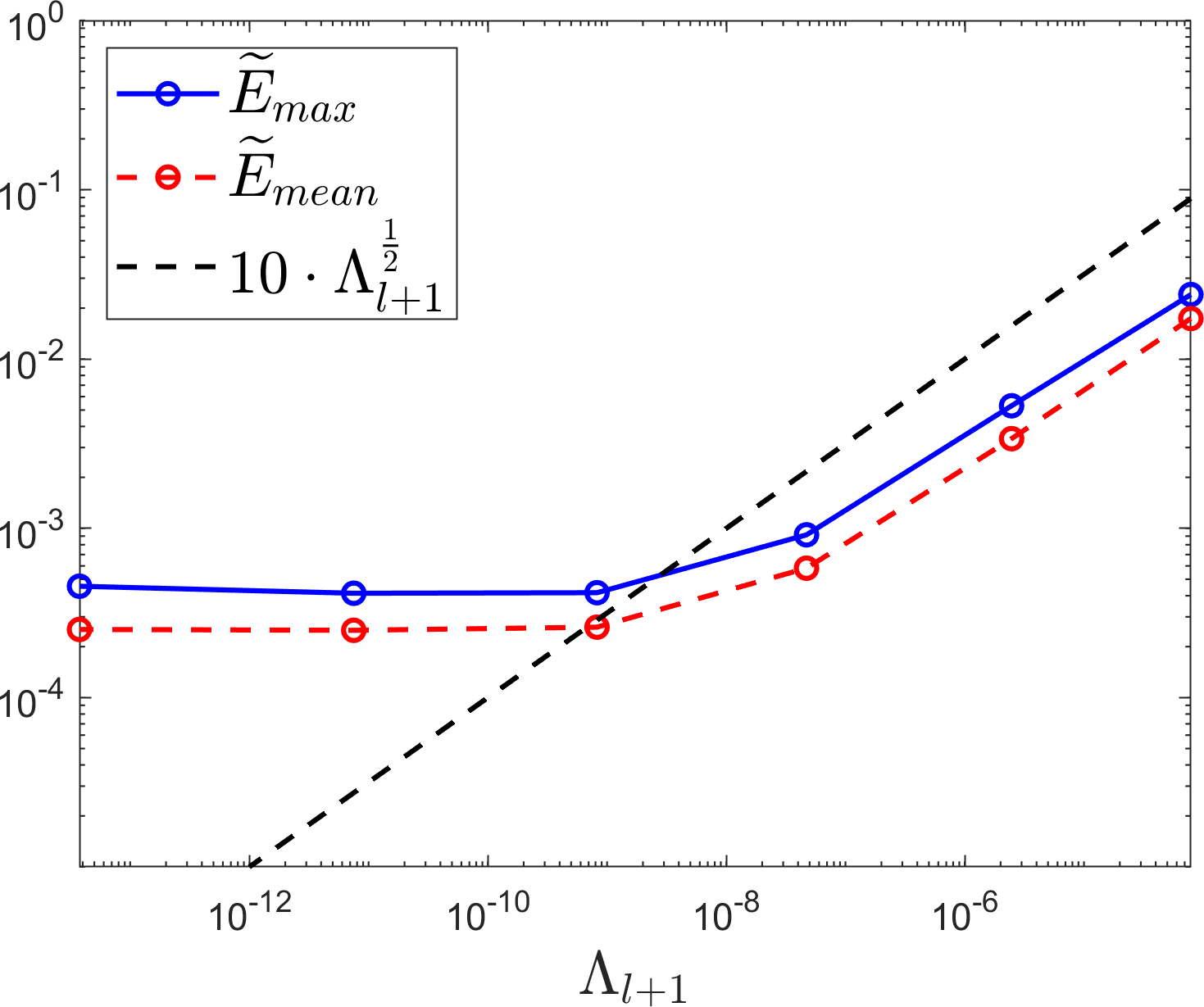}
\end{tabular}
\end{center}
\caption{Error estimates $\widetilde{E}_{\max}$ (solid blue line with circles) and $\widetilde{E}_{\text{mean}}$ 
(dashed red line with circles) as functions of $\eps$ (left), $\delta$ (middle) and $\Lambda_{l+1}$ (right).
The slopes of dashed black lines represent the theoretical bounds from the right-hand side of \eqref{eqn:err1}:
$\eps$ (left), $\delta^2$ (middle) and $\Lambda_{l+1}^{1/2}$ (right).}
\label{fig:erradvdiff}
\end{figure}

We employ the same numerical error estimates \eqref{eqn:erra}--\eqref{eqn:errestmean} as in the previous 
example with $\Delta t = 1/60$ and $N = 60$. The error estimates \eqref{eqn:errestmax}--\eqref{eqn:errestmean}
are computed over a test set  $\widetilde{\mathcal{A}}$ of $\widetilde{K} = 500$ parameter samples
chosen randomly from $\mathcal{A}$. We fix $\Delta t$ and $h = 0.05$
and display in Figure~\ref{fig:erradvdiff} error estimates $\widetilde{E}_{\max}$ and $\widetilde{E}_{\text{mean}}$
for $\eps$, $\delta$ and $\Lambda_{l+1}$ that vary in the following ranges:
$\eps \in [10^{-6}, 10^{-1}]$, $\delta \in [1/30, 1/50]$, and 
$\Lambda_{l+1} \in [3 \cdot 10^{-14}, 8 \cdot 10^{-5}]$ (which correspond to $\ell=\{5,7,9,11,13,15\}$). The recovered first TT-rank (i.e. the universal space dimension)
for $\eps=10^{-i}$, {\small $i=1,\dots,6$}, was $\widetilde R_1=\{7,20,40,67,103,159\}$, while the local {LRTD-ROM} space dimension is $\ell$. We observe in Figure~\ref{fig:erradvdiff} 
that error estimates $\widetilde{E}_{\max}$ and $\widetilde{E}_{\text{mean}}$ with respect to 
$\eps$ and $\Lambda_{l+1}$ behave as predicted by \eqref{eqn:err1}, until the two smallest values of $\eps$ 
and $\Lambda_{l+1}$ where they flatten out. The reason for such behavior is that the error becomes dominated by 
the contribution of $\delta = 1/50$. Taking a smaller value of $\delta$ for this example, given its relatively large number of parameters, proves to be computationally infeasible in our current implementation. This is due to the resulting dense sampling set $\widetilde{\mathcal{A}}$, which, in turn, leads to an impractically large snapshot tensor $\bPhi$. The size of $\bPhi$ becomes problematic both in terms of storage and the computational time required to populate it with snapshots. To handle finer parameter space meshes, the {LRTD-ROM} implementation based  on  a sparse sampling of the parameter domain  is required (cf. Remark~\ref{Rem:Impl}). Our current research is focused on extending existing low-rank tensor interpolation and completion methods to accommodate this requirement.


\section*{Acknowledgments} 
The authors were supported in part by the U.S. National Science Foundation under award DMS-2309197.
This material is based upon research supported in part by the U.S. Office of Naval Research under award number N00014-21-1-2370 to Mamonov. 

\bibliographystyle{siamplain}
\bibliography{literatur}{}

\end{document}